\documentclass[a4paper,11pt]{article}

\usepackage{epsfig}
\usepackage{amsfonts}
\usepackage{authblk}
\usepackage{amssymb}
\usepackage{float}
\usepackage{blkarray}
\usepackage{color}
\usepackage{amsmath,bm}
\usepackage{mathabx}
\usepackage{mathtools}
\usepackage{nicefrac}
\usepackage{a4wide}
\usepackage{array}
\usepackage{multicol}
\usepackage{pdfpages}
\usepackage{subfig}
\usepackage{multirow}
\usepackage{enumitem}
\usepackage{bbold}
\usepackage{xparse}								
\usepackage{stmaryrd}							
\usepackage{sidecap}
\usepackage{bbm}
\usepackage{pgfplots}
\pgfplotsset{compat=1.16}

\usepackage{booktabs}

\usepackage{subcaption}
\usepackage{caption}

\usepackage[bbgreekl]{mathbbol}
\usepackage{amsfonts}

\usepackage{graphicx}        
\usepackage{multicol}        
\usepackage[bottom]{footmisc}
\usepackage{authblk}
\usepackage{float,cancel}
\usepackage{mathtools,amsthm,amsfonts}

\usepackage{color}

\usepackage{bm}
\usepackage{nicefrac}
\usepackage[hmargin={29mm,29mm},vmargin={30mm,35mm}]{geometry}
\usepackage{array}
\usepackage{pdfpages}
\usepackage{subfig}
\usepackage{multirow}
\usepackage{enumitem}
\usepackage{bbold}
\usepackage{stmaryrd}							
\usepackage{tcolorbox}
\usepackage{algorithmic}
\usepackage[ocgcolorlinks,allcolors={blue}]{hyperref}
\usepackage{url}
\usepackage[bbgreekl]{mathbbol}

\DeclareSymbolFontAlphabet{\mathbb}{AMSb}
\DeclareSymbolFontAlphabet{\mathbbl}{bbold}

\def\R{\mathbb R}

\def\D{\mathcal{D}}

\def\XD{X_{\D}}

\def\PiDm{\Pi^m_{\D}}
\def\PiDf{\Pi^f_{\D}}
\def\PiDbullet{\Pi^{\bullet}_{\D}}
\def\gradDm{\nabla^m_{\D}}
\def\gradDf{\nabla^f_{\D}}

\def\<{\langle}
\def\>{\rangle}

\def\Chi{\raise .3ex \hbox{\large $\chi$}}

\def\ol#1{\overline{#1}}

\def\({\Bigl (}
\def\){\Bigr )}
\def\dsp{\displaystyle}
\def\x{{\bf x}}

\def\n{{\bf n}}

\def\U{{\bf U}}

\def\q{{\bf q}}

\def\aa{{\mathfrak a}}

\def\M{\mathbf{M}}
\def\L{\boldsymbol{\Lambda}}
\def\l{\boldsymbol{\lambda}}
\def\m{\boldsymbol{\mu}}

\def\cells{\mathcal{M}}
\def\faces{\mathcal{F}}

\def\edges{\mathcal{E}}

\def\Ens{\mathfrak{E}_s}

\def\Tref{{T}_{\rm ref}}
\def\pref{{p}_{\rm ref}}
\def\rhoref{{\varrho}_{\rm ref}}

\def\x{{\bf x}}
\def\dsp{{\displaystyle x}}
\def\d{{\rm d}}
\def\div{{\rm div}}

\def\bq{\mathbf{q}}
\def\K{\mathbb{K}}
\def\n{\mathbf{n}}
\def\dsp{\displaystyle}
\def\bu{\mathbf{u}}
\def\bv{\mathbf{v}}

\def\bV{\mathbf{V}}
\def\Id{\mathbb{I}}
\def\bF{\mathbf{F}}

\def\df{\mathrm{d}_f}
\def\dfsig{\mathrm{d}_{f,\sigma}}
\def\dfD{\mathrm{d}_{f,\D}}
\def\dzero{\mathrm{d}_0}

\def\bbsig{\bbsigma}
\def\bbeps{\bbespilon}

\newcommand{\jump}[1]{\llbracket #1 \rrbracket}

\newcommand{\dtn}{{\Delta t^n}}

\newcommand{\deln}{\delta_t^{n}}

\newcommand{\delnbullet}{\delta_{t,\bullet}^{n}}
\newcommand{\delnm}{\delta_{t,m}^{n}}
\newcommand{\delnf}{\delta_{t,f}^{n}}

\newcommand{\email}[1]{\href{mailto:#1}{#1}}

\begingroup
    \makeatletter
    \@for\theoremstyle:=definition,remark,plain\do{%
        \expandafter\g@addto@macro\csname th@\theoremstyle\endcsname{%
            \addtolength\thm@preskip\parskip
            }%
        }
\endgroup

\newtheorem{theorem}{Theorem}
\numberwithin{theorem}{section}
\newtheorem{proposition}[theorem]{Proposition}
\newtheorem{lemma}[theorem]{Lemma}

\theoremstyle{remark}

\theoremstyle{definition}

\def\trace{\gamma}

\def\del{\partial}
\def\div{{\rm div}}

\definecolor{labelkey}{rgb}{0.6,0,1}

\usepackage[normalem]{ulem}
 \newcounter{corr}
 \definecolor{violet}{rgb}{0.580,0.,0.827}
 \newcommand{\corr}[3]{\typeout{Warning : a correction remains in page
 \thepage}
 				\stepcounter{corr}        
 				{\color{blue}\ifmmode\text{\,\sout{\ensuremath{#1}}\,}\else\sout{#1}\fi}
         {\color{red}#2}
         {\color{violet} \ifmmode\text{#3}\else #3\fi }}

\usepackage{parskip}
\usepackage{algorithmic}

\newcommand{\logLogSlopeTriangle}[5]
{

    \pgfplotsextra
    {
        \pgfkeysgetvalue{/pgfplots/xmin}{\xmin}
        \pgfkeysgetvalue{/pgfplots/xmax}{\xmax}
        \pgfkeysgetvalue{/pgfplots/ymin}{\ymin}
        \pgfkeysgetvalue{/pgfplots/ymax}{\ymax}

        \pgfmathsetmacro{\xArel}{#1}
        \pgfmathsetmacro{\yArel}{#3}
        \pgfmathsetmacro{\xBrel}{#1-#2}
        \pgfmathsetmacro{\yBrel}{\yArel}
        \pgfmathsetmacro{\xCrel}{\xArel}

        \pgfmathsetmacro{\lnxB}{\xmin*(1-(#1-#2))+\xmax*(#1-#2)} 
        \pgfmathsetmacro{\lnxA}{\xmin*(1-#1)+\xmax*#1} 
        \pgfmathsetmacro{\lnyA}{\ymin*(1-#3)+\ymax*#3} 
        \pgfmathsetmacro{\lnyC}{\lnyA+#4*(\lnxA-\lnxB)}
        \pgfmathsetmacro{\yCrel}{\lnyC-\ymin)/(\ymax-\ymin)} 

        \coordinate (A) at (rel axis cs:\xArel,\yArel);
        \coordinate (B) at (rel axis cs:\xBrel,\yBrel);
        \coordinate (C) at (rel axis cs:\xCrel,\yCrel);

     \draw[#5]   (A)-- 
                       (B)-- 
                       (C)-- node[pos=0,anchor=west] {#4}
                       cycle;

    }
    
}

\usepackage{diagbox}
\usepackage{booktabs}

\makeatletter
\def\thm@space@setup{%
  \thm@preskip=\parskip \thm@postskip=0pt
}
\makeatother

\newcommand{\doi}[1]{\href{https://dx.doi.org/#1}{DOI:#1}}

\begin{document}

\usetikzlibrary{calc}

\title{Discretisations of mixed-dimensional Thermo-Hydro-Mechanical models preserving energy estimates}
\author[1,3]{{J\'er\^ome Droniou}\footnote{\email{jerome.droniou@umontpellier.fr}}}
\author[2]{{Mohamed Laaziri}\footnote{\email{mohamed.laaziri@univ-cotedazur.fr}}}
\author[2]{{Roland Masson}\footnote{\email{roland.masson@univ-cotedazur.fr}}}
\affil[1]{IMAG, Universit\'e de Montpellier, CNRS, Montpellier, France}%
\affil[2]{Universit\'e C\^ote d'Azur, Inria, CNRS, Laboratoire J.A. Dieudonn\'e, team Coffee, Nice, France}%
\affil[3]{School of Mathematics, Monash University, Victoria 3800, Australia}%

\date{}
\maketitle

\begin{abstract}
  In this study, we explore mixed-dimensional Thermo-Hydro-Mechanical (THM) models in fractured porous media accounting for Coulomb frictional contact at matrix fracture interfaces. The simulation of such models plays an important role in many applications such as hydraulic stimulation in deep geothermal systems and assessing induced seismic risks in CO$_2$ storage.
We first extend to the mixed-dimensional framework the thermodynamically consistent THM models derived in \cite{coussy} based on first and second principles of thermodynamics. Two formulations of the energy equation will be considered  based either on energy conservation or on the entropy balance, assuming a vanishing thermo-poro-elastic dissipation. 
Our focus is on space time discretisations preserving energy estimates for both types of formulations and for a general single phase fluid thermodynamical model. This is achieved by a Finite Volume discretisation of the non-isothermal flow based on coercive fluxes and a tailored discretisation of the non-conservative convective terms. It is combined with a mixed Finite Element formulation of the contact-mechanical model with face-wise constant Lagrange multipliers accounting for the surface tractions, which preserves the dissipative properties of the contact terms. 
The discretisations of both THM formulations are investigated and compared in terms of convergence, accuracy and robustness on 2D test cases. It includes a Discrete Fracture Matrix model with a convection dominated thermal regime, and either a weakly compressible liquid or a highly compressible gas thermodynamical model. 
\noindent
~
\bigskip \\
\textbf{Keywords:} Thermo-Hydro-Mechanical (THM) model, Mixed-dimensional model, Discrete Fracture Matrix model, Contact-mechanics, Thermodynamically consistent discretisation, Finite Volume, Mixed Finite Element formulation. 
\end{abstract}
%
\section{Introduction}
Thermo-Hydro-Mechanical (THM) models in fractured/faulted porous rocks play an important role in addressing the challenges of a sustainable exploitation of subsurface resources. This is for example typically the case of deep geothermal energy production and CO$_2 $ geological storage. THM models offer valuable insights into the complex interactions between temperature changes, fluid flow, rock deformation and fracture/fault mechanical behavior within the subsurface. They play a pivotal role in assessing potential risks, informing mitigation strategies and managing hydraulic stimulation of geothermal systems. Similarly, in CO$_2$ sequestration projects, THM models are instrumental in predicting and preventing issues like fault reactivation which can  potentially induce CO$_2$ leakage or seismicity.

The THM models considered in this work initially integrate a mixed-dimensional approach, coupling a non-isothermal Poiseuille flow within the network of fractures/faults represented as co-dimension one surfaces, to the non-isothermal Darcy flow in the surrounding porous rock, known as the matrix. Let us refer to \cite{MAE02,FNFM03,KDA04,MJE05,BGGLM16,BHMS2016,FLEMISCH2018239,NBFK2019,BERRE2021103759} and the references there-in for the derivation and for various discretisations of such mixed-dimensional models in the context of single phase Darcy flows. 
The second key component is the thermo-poro-mechanical model coupling the deformation of the rock with the non-isothermal Darcy flow within the matrix domain. Let us refer to the monograph \cite{coussy} as a key reference textbook on a thermodynamically consistent derivation of such models.  In this work, we will follow this framework further assuming small strains and porosity variations as well as a linear thermo-poro-elastic behavior of the porous rock. 
The third ingredient is related to the mechanical behavior of the fractures/faults typically based on contact-mechanics taking into account frictional contact at matrix-fracture interfaces.   
This type of mixed-dimensional poromechanical models have been the object of many recent works both in the isothermal case \cite{NEJATI2016123,GKT16,contact-norvegiens,tchelepi-castelletto-2020,GDM-poromeca-cont,GDM-poromeca-disc,BDMP:21,BoonNordbotten22} and in the non-isothermal case \cite{SALIMZADEH2018212,GARIPOV2019104075,STEFANSSON2021114122}. 

Building upon existing research, our work focuses on discretisations of mixed-dimensional THM models preserving energy estimates for a general single phase fluid thermodynamical model, an aspect not thoroughly explored in previous studies. We first extend to the mixed-dimensional framework the thermodynamically consistent THM models derived in \cite{coussy} based on first and second principles of thermodynamics. Two formulations of the energy equation will be considered  based either on energy conservation or on the entropy balance, assuming a vanishing thermo-poro-elastic dissipation.  The entropy balance formulation is frequently combined with a small Darcy velocity assumption and a linearisation of the Fourier term based on a small temperature variation assumption \cite{BRUN20201964,both2019gradient}. It leads to an approximate entropy equation which will be investigated in this work both theoretically and numerically from the points of view of energy estimates and accuracy of the solution. One difficulty that needs to be dealt with when starting from such entropy balance equation is the ability to return back to the energy equation at the discrete level. This will be addressed in this work by preserving at the discrete level the links between the energy conservation and entropy balance equations.

To achieve these goals, the space and time discretisations of the non-isothermal flow and of the contact-mechanics must be selected carefully. Regarding the flow, our framework is based on coercive fluxes and incorporate a possible upwind approximation of the convection terms in order to deal with convection dominated regimes. The discretisation of the non-conservative convective terms in the entropy balance formulation is designed in order to preserve the link with the energy formulation. Although, this setting accounts for a large class of coercive Finite Volume schemes, we focus in the following on the Hybrid Finite Volume (HFV) discretisation \cite{droniou2010unified,BGGLM16} in order to simplify the presentation. The time integration scheme is based on a forward Euler discretisation including a semi-implicit variation in the entropy balance approach.  
Regarding the discretisation of the mechanical model, it will be based, for the sake of simplicity, on a conforming Finite Element approximation of the displacement field accounting for its discontinuity at matrix fracture interfaces. The discretisation of the frictional contact along the fracture network is a key feature both for the robustness of the scheme and for the derivation of energy estimates.  
Different formulations such as mixed or stabilized mixed formulations \cite{haslinger-96,Wohlmuth11,Lleras-2009}, augmented Lagrangian \cite{BHL2023} and Nitsche methods \cite{Chouly2017,CHLR2020,beaude2023mixed} have been developed to discretise Coulomb frictional contact. 
In line with \cite{Hild17,tchelepi-castelletto-2020,BDMP:21}, a mixed formulation with face-wise constant Lagrange multipliers is selected. This choice maintains at the discrete level the dissipative properties of the contact terms  and the lower bound on the fracture aperture. It enables the straightforward handling of complex fracture networks, including corners, tips, and intersections, and lead to local expressions of the contact equations, making possible the use of efficient semi-smooth Newton nonlinear solvers. 

The rest of this article is organised as follows. Section \ref{sec:thmmodel} presents the mixed-dimensional THM models with energy equations based either on an energy conservation formulation or on an approximate entropy balance formulation. The energy estimates satisfied by both models are formally derived in  Section \ref{sec:energyestimates}. Their discretisations are described in Section \ref{sec:discretisation} including the Finite Volume discretisations of the non-isothermal flow in Section \ref{sec:discretemassenergy} and the mixed formulation of the contact-mechanics in Section \ref{sec:discretecontactmechanics} . The discrete energy estimates are derived in Section \ref{sec:discreteenergyestimates}  for the discretisations of both models which are compared numerically  in Section \ref{sec:numerics} in terms of convergence, accuracy and robustness. We first consider in Section \ref{sec:test1}  a manufactured solution for an incompressible fluid and different Peclet numbers on a 2D square domain without fractures. Then, a 2D Discrete Fracture Matrix (DFM) model  with a six fracture network is investigated in Section \ref{sec:testdfm} both for the case of a weakly compressible liquid and for the case of a highly compressible perfect gas.

\section{Mixed-dimensional Thermo-Hydro-Mechanical models}\label{sec:thmmodel}

\subsection{Mixed-dimensional geometry and function spaces}

In what follows, scalar fields are represented by lightface letters, vector fields by boldface letters.
We let $\Omega\subset\R^d$, $d\in\{2,3\}$, denote a bounded polytopal domain, partitioned
into a fracture domain $\Gamma$ and a matrix domain $\Omega\backslash\overline\Gamma$.
The network of fractures is defined by 
$$
\overline \Gamma = \bigcup_{i\in I} \overline \Gamma_i,
$$  
where each fracture $\Gamma_i\subset \Omega$, $i\in I$, is a planar and simply connected polygonal domain, which is relatively open in the hyperplane it spans. Without restriction of generality, we will assume that the fractures may only intersect at their boundaries (Figure \ref{fig:network}), that is, for any $i,j \in I, i\neq j$ it holds $\Gamma_i\cap \Gamma_j = \emptyset$, but not necessarily $\overline{\Gamma}_i\cap \overline{\Gamma}_j = \emptyset$.

\begin{figure}[h!]
\begin{center}
\includegraphics[scale=.55]{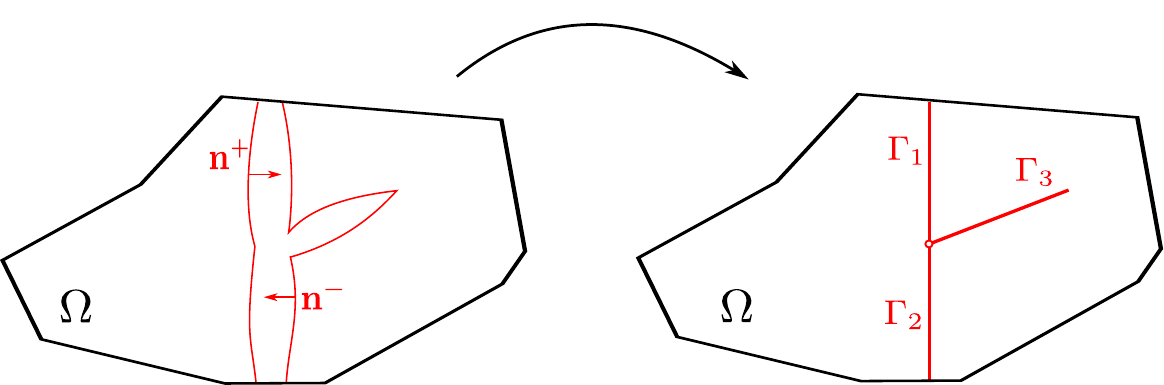}   
\caption{Illustration of the dimension reduction in the fracture aperture for a 2D domain $\Omega$ with three intersecting fractures $\Gamma_i$, $i\in\{1,2,3\}$, with the equi-dimensional geometry on the left and the mixed-dimensional geometry on the right.}
\label{fig:network}
\end{center}
\end{figure}

The two sides of a given fracture of $\Gamma$ are denoted by $\pm$ in the matrix domain, with unit normal vectors $\n^\pm$ oriented outward from the sides $\pm$. We denote by $\gamma_\aa$ the trace operators on the side $\aa \in \{+,-\}$ of $\Gamma$ for functions in $H^1(\Omega{\setminus}\overline\Gamma)$ and by $\gamma_{\del\Omega}$ the trace operator for the same functions on $\del\Omega$. The jump operator on $\Gamma$ for functions $\bu$ in  $(H^1(\Omega\backslash\overline\Gamma))^d$ is defined by
$$
\jump{\bu} = {\gamma_+ \bu - \gamma_- \bu}, 
$$
and we denote by
$$
\jump{\bu}_\n = \jump{\bu}\cdot\n^+ \quad \mbox{ and } \quad \jump{\bu}_\tau = \jump{\bu} -\jump{\bu}_\n \n^+
$$
its normal and tangential components. The notation $\jump{\q_m}_\n$ will also be used to denote the normal jump of  flux functions $\q_m \in H_{\div}(\Omega\setminus\overline\Gamma)$ defined by
$$
\jump{\q_m}_\n = \gamma_\n^+ \q_m + \gamma_\n^- \q_m
$$
with $\gamma_\n^\aa$ the normal trace operator on the side $\aa$ of $\Gamma$ oriented outward to the side $\aa \in \{+,-\}$. The notation $\gamma_\n^\aa$ will also be applied to tensor fields. 
The tangential gradient and divergence along the fractures are respectively denoted by $\nabla_\tau$ and $\div_\tau$. The symmetric gradient operator $\bbeps$ is defined such that $\bbeps(\bv) = {1\over 2} (\nabla \bv +(\nabla \bv)^t)$ for a given vector field $\bv\in H^1(\Omega\backslash\overline\Gamma)^d$.

To fix ideas, 
the pressure $p$ and temperature $T$ are assumed to be continuous at the matrix--fracture interface, that is, these functions belong to $H^1(\Omega)$. As a consequence, letting $\gamma:H^1(\Omega)\to L^2(\Gamma)$ be the trace operator on the fracture, the fracture pressure and temperature are given by $p_f=\gamma p$ and $T_f=\gamma T$.

The space for the displacement field is defined by 
\begin{equation*}
\U_0 =\{ \bv \in H^1(\Omega\backslash\overline\Gamma)^d : \trace_{\del\Omega} \bv = 0\}. 
\end{equation*}

\subsection{Mixed-dimensional models}\label{sec:mixeddim} 

We consider a Thermo-Hydro-Mechanical (THM) model under the hypothesis of small perturbations for the skeleton accounting for small strain, displacement and variations of porosity \cite{coussy}. Linear isotropic thermo-poro-elastic constitutive laws are considered for the skeleton assuming small variations of temperature around the reference temperature $\Tref$. The Darcy law is used for the fluid velocity and the Fourier law for the thermal conduction. Thermal equilibrium is assumed between the fluid and the skeleton, 
and the mechanical inertial term is modelled using the frozen specific average fluid-rock density $m_0$.

The fluid thermodynamical properties, depending on its pressure $p$ and temperature $T$, are
\begin{itemize}
\item $\varrho(p,T)$: specific density,
\item $e(p,T)$: specific internal energy,
\item $h(p,T) = e(p,T) + {p\over \varrho(p,T)}$: specific enthalpy, and
\item $\eta(p,T)$: dynamic viscosity.
\end{itemize}
For a fluid property $\Xi = \varrho, e, h, \eta$, we use the short notations  $\Xi_m = \Xi(p,T)$ in the matrix and $\Xi_f = \Xi (p_f, T_f) = \gamma \Xi_m$ along the fracture network.  

The primary unknowns of the model are the fluid pressure $p$, the fluid temperature $T$ and the skeleton displacement field $\bu$. They are solutions of the nonlinear system of PDEs coupling the mixed-dimensional fluid mass conservation equation, the mixed-dimensional total energy conservation equation, the skeleton momentum balance equation, and the frictional contact conditions at matrix fracture interfaces. The model is detailed below starting with the equations in the matrix followed by the equations along the fracture network. Two different formulations of the energy equation will be considered.

\subsubsection{Matrix model}

In the matrix, the model accounts for the mass and energy conservation equations coupled to the skeleton momentum balance equation. Letting $t_F$ be the final simulation time, we therefore consider
\begin{subequations}
\label{eq:matrix}
\begin{alignat}{2}
\partial_t (\varrho_m \phi)+\div(\varrho_m \bV_m)={}& G_m &&\quad\mbox{ in $(0,t_F)\times\Omega$},\label{eq:mass.matrix}\\
T \partial_t  S_s + p \partial_t \phi + \partial_t (\varrho_m\phi e_m) + \div(\varrho_m h_m\bV_m + \q_m ) ={}& H_m &&\quad\mbox{ in $(0,t_F)\times\Omega$},\label{eq:energy.h.matrix}\\
m_0 \partial_t^2 \bu -\div\bbsig ={}& \bF  &&\quad\mbox{ in $(0,t_F)\times\Omega$},\label{eq:meca}
\end{alignat}
with the Darcy and Fourier laws governing respectively the matrix fluid velocity and the matrix conductive thermal flux:  
\begin{alignat}{2}
&\bV_m=-\frac{\K(\phi)}{\eta_m} \nabla p,  \quad \q_m = -\Lambda_m(\phi)\nabla T, \label{eq:DarcyFourier.matrix}
\end{alignat}
\end{subequations}
where $\K$ and $\Lambda_m$ are respectively the rock permeability and the fluid rock average thermal conductivity, both possibly depending on the matrix porosity $\phi$. 
The energy equation \eqref{eq:energy.h.matrix} has been obtained under the hypothesis of reversible mechanical deformation in the sense of zero thermo-poro-mechanical dissipation:
\begin{equation}
T \partial_t S_s + p \partial_t \phi + \bbsig : \partial_t \bbeps(\bu) - \partial_t E_s = 0. \label{eq:revdef}
\end{equation}
Assuming a linear isotropic thermo-poro-elastic behavior of the skeleton, this gives the following constitutive laws
\begin{subequations}
\label{eq:laws}
\begin{alignat}{2}
\partial_t \phi ={}& b ~\div \partial_t\bu - \alpha_\phi~\partial_t T + \frac1N\partial_t p\,,\label{eq:dtphi}\\
\partial_t S_s={}& \alpha_s K_s ~\div\partial_t\bu - \alpha_\phi ~\partial_t p +\frac{C_s}{\Tref}\partial_tT \label{eq:dtSs},\\
\bbsig ={}& \bbsig^e(\bu) - b ~p \Id - \alpha_s K_s (T-\Tref) \Id \label{eq:sigmaT},\\
\bbsig^e(\bu) ={}&  2 \mu  ~\bbeps(\bu)+ \lambda ~\div \bu ~\Id \label{eq:sigmae},
\end{alignat}
\end{subequations}
which derive from the following volumetric skeleton internal energy  \cite{coussy}: 
\begin{equation}
E_s = \mu |\bbeps(\bu)|^2 + {\lambda \over 2} (\div(\bu))^2  +  \frac12 {} \begin{bmatrix} p & T \end{bmatrix} M \begin{bmatrix} p\\T \end{bmatrix} + \alpha_s K_s \Tref ~\div\bu, \label{eq:rockenergy}
\end{equation}
where the matrix $M$, defined below, is assumed to be definite positive:
\begin{equation*}\label{eq:M.sdp}
  M:= \begin{bmatrix} \frac{1}{N} & -\alpha_\phi \\-\alpha_\phi & \frac{C_s}{\Tref} \end{bmatrix}.
\end{equation*}

In the equations above, $S_s$ is the volumetric skeleton entropy, $\bbsig$ the total stress tensor and $\bbsig^e$ the effective stress tensor. 
The parameters $\mu$ and $\lambda$ are the effective Lame coefficients, 
$N$ is the Biot modulus, $b$ the Biot coefficient, 
$K_s$ is the bulk modulus, $\alpha_s$ is the volumetric skeleton thermal dilation coefficient,
$\alpha_\phi$ is the volumetric thermal dilation coefficient related to the porosity, and 
$C_s$ is the skeleton volumetric heat capacity. 

Considering the combination ${1\over T}\times$ \eqref{eq:energy.h.matrix} $- {h_m\over T}\times$ \eqref{eq:mass.matrix} leads to the following alternative formulation of the energy equation
\begin{equation}\label{eq:entro.matrix}
\partial_t S_s+
\frac{\varrho_m \phi}{T}\partial_t e_m + \frac{p\varrho_m}{T}{} \phi \partial_t \frac{1}{\varrho_m}
  + \frac{1}{T}\varrho_m\bV_m\cdot\nabla h_m + \frac{1}{T}\div\q_m= {H_{m} \over T} - {h_m \over T} G_{m}. 
\end{equation}
An approximation of \eqref{eq:entro.matrix} is obtained based on the classical assumptions of small Darcy velocity $\bV_m$  and small variations of temperature around $\Tref$. Then, using the approximation  
\begin{align*}
  \varrho_m\bV_m\cdot\nabla h_m  &=  \varrho_m\bV_m\cdot\nabla e_m + p \varrho_m\bV_m\cdot\nabla {1 \over \varrho_m} + \bV_m\cdot\nabla p \\
  & \sim \varrho_m\bV_m\cdot\nabla e_m + p \varrho_m\bV_m\cdot\nabla {1 \over \varrho_m}, 
\end{align*}
and the linearisation ${1 \over T} \div \q_m \sim {1 \over \Tref} \div \q_m$, 
we obtain the following approximate equation: 

\begin{equation}\label{eq:entro.app.matrix}
\begin{aligned}
\partial_t S_s+
\frac{\varrho_m \phi}{T}\partial_t e_m + \frac{p\varrho_m}{T}{}&\phi \partial_t \frac{1}{\varrho_m}
  + \frac{1}{T}\varrho_m\bV_m\cdot\nabla e_m\\
  & + \frac{p}{T}\varrho_m\bV_m\cdot\nabla {1 \over \varrho_m} + \frac{1}{\Tref}\div\q_m= {H_{m} \over T} - {h_m \over T} G_{m}. 
\end{aligned}
\end{equation}
Note that, using $Tds = de + p d{1 \over \varrho}$ with $s$ the fluid specific entropy, equation \eqref{eq:entro.app.matrix} becomes 
$$
\begin{aligned}
\partial_t S_s+ \varrho_m \phi \partial_t s_m 
  + \varrho_m\bV_m\cdot\nabla s_m  + \frac{1}{\Tref}\div\q_m= {H_{m} \over T} - {h_m \over T} G_{m}.
\end{aligned}
$$
This equation is a classical approximate non conservative formulation of the entropy equation \cite{coussy}, which motivates the terminology adopted in the following of \emph{approximate entropy equation} for \eqref{eq:entro.app.matrix} and of \emph{entropy equation} for \eqref{eq:entro.matrix}.
\subsubsection{Fracture model} 

The mass and energy conservation equations of the reduced fracture model are obtained by integration along the fracture width of the equi-dimensional equations taking into account the mass and energy normal flux continuity at matrix fracture interfaces. This process leads to
\begin{subequations}
\label{eq:fracture}
\begin{equation}
\partial_t (\varrho_f \df)+\div_\tau(\varrho_f  \bV_f)-\jump{\varrho_m\bV_m}_\n = G_{f}\quad\mbox{ in $(0,t_F)\times\Gamma$},\label{eq:mass.fracture}
\end{equation}
\begin{equation}
\begin{aligned}
p_f \partial_t d_f + \partial_t (\varrho_f \df e_f) + \div_\tau(\varrho_f h_f\bV_f + \q_f) - \jump{\varrho_m h_m\bV_m + \q_m}_\n
&= H_{f}\quad\mbox{ in $(0,t_F)\times\Omega$},
\end{aligned}
\label{eq:energy.h.fracture}
\end{equation}
with the fluid tangential velocity and thermal conductive flux integrated along the fracture width defined by 
\begin{equation}
\bV_f=-\frac{C_f(d_f)}{\eta_f} \nabla_\tau p_f\,,\quad\q_f = -  \Lambda_f(d_f)\nabla_\tau T_f\,,
\label{eq:darcy.fourier.fracture}
\end{equation}
and where 
\begin{equation}
\df= \dzero - \jump{\bu}_\n
\label{eq:dt.df}
\end{equation}
is the fracture aperture with $\dzero$ the aperture at contact state as illustrated in  Figure \ref{fig:d0}. 
In  \eqref{eq:darcy.fourier.fracture}, $C_f$ is the fracture hydraulic conductivity typically given by the Poiseuille law $C_f = {(\df)^3 \over 12}$, and $\Lambda_f$ is the fracture thermal conductivity also possibly depending on $\df$.  
At matrix fracture interfaces, a contact Coulomb frictional model is  considered defined as follows 
\begin{equation}
\label{eq:meca.contact} 
\left\{\!\!\!\!
\begin{array}{lll}
  & {\bf T}^{+} + {\bf T}^{-} = {\bf 0}  & \mbox{ on } (0,t_F)\times \Gamma,\\[1ex]
  & T_n \leq 0, \,\,   \jump{\bu}_\n \leq 0, \,\, \jump{\bu}_\n ~  T_n = 0   & \mbox{ on } (0,t_F)\times \Gamma,\\[1ex]
  & |{\bf T}_\tau| \leq - F ~T_n   & \mbox{ on } (0,t_F)\times \Gamma, \\[1ex]
  &  (\partial_t \jump{\bu}_\tau)\cdot {\bf T}_\tau - F ~T_n |\partial_t \jump{\bu}_\tau|=0   & \mbox{ on } (0,t_F)\times \Gamma,
\end{array}
\right.
\end{equation}
where the vectorial surface tractions and their normal and tangential components are defined by 
\begin{equation*}
\left\{\!\!\!\!
\begin{array}{lll}
  & {\bf T}^{\aa} = \gamma_\n^\aa {\bbsig(\bu) +  p_f \n^\aa}  & \mbox{ on } (0,t_F)\times \Gamma, \ \aa \in \{+,-\},\\[1ex]
  & T_n = {\bf T}^{+}\cdot \n^+ & \mbox{ on } (0,t_F)\times \Gamma, \\[1ex]
  & {\bf T}_\tau = {\bf T}^{+} - ({\bf T}^{+}\cdot \n^+)\n^+ & \mbox{ on } (0,t_F)\times \Gamma.  
\end{array}
\right.
\end{equation*}
\label{eq:model.fracture}
\end{subequations}
In a similar way as in the matrix, we also consider the entropy equation defined by the combination ${1\over T}\times$ \eqref{eq:energy.h.fracture} $- {h_m\over T}\times$ \eqref{eq:mass.fracture} leading to 
\begin{equation}\label{eq:entro.fracture_v1}
\begin{aligned}
\frac{\varrho_f \df}{T_f}\partial_t e_f + \frac{p_f\varrho_f}{T_f}{}& \df \partial_t \frac{1}{\varrho_f}
  + \frac{1}{T_f}\varrho_f\bV_f\cdot\nabla h_f  + \frac{1}{T_f}( \div\q_f - \jump{\q_m}_\n)\\
  & - {1 \over T_f} \(\jump{\varrho_m h_m \bV_m}_\n - h_f \jump{\varrho_m \bV_m}_\n\)  = {H_{f} \over T_f} - {h_f \over T_f} G_{f}.
\end{aligned}
\end{equation}

By continuity of the pressure and temperature (and thus of the fluid specific enthalpy), $\jump{\varrho_m h_m\bV_m}_\n- h_f \jump{\varrho_m\bV_m}_\n=0$ and the last term in the left-hand side of \eqref{eq:entro.fracture_v1} therefore vanishes. However, at the discrete level, as a result of a possible upwinding this compensation may not necessarily occur. To keep that in mind in the continuous model, and justify the discretisation chosen in Section \ref{sec:discretisation}, we therefore write this vanishing term in the form of a (fictitious) jump of the enthalpy between the matrix and the fracture. 
$$
\jump{\varrho_m h _m\bV_m}_\n- h_f \jump{\varrho_m\bV_m}_\n = \sum_{\aa \in \{+,-\}} \gamma_\n^\aa (\varrho_m \bV_m)  (\gamma^\aa h_m - h_f ). 
$$
Therefore, \eqref{eq:entro.fracture_v1} can be recast as

\begin{equation}\label{eq:entro.fracture}
\begin{aligned}
\frac{\varrho_f \df}{T_f}\partial_t e_f + \frac{p_f\varrho_f}{T_f}{}& \df \partial_t \frac{1}{\varrho_f}
  + \frac{1}{T_f}\varrho_f\bV_f\cdot\nabla h_f  + \frac{1}{T_f}( \div\q_f - \jump{\q_m}_\n)\\
  & - {1 \over T_f} \sum_{\aa \in \{+,-\}} \gamma_\n^\aa(\varrho_m \bV_m )  (\gamma^\aa h_m - h_f )  = {H_{f} \over T_f} - {h_f \over T_f} G_{f}.
\end{aligned}
\end{equation}
Its approximation based on small Darcy velocity and temperature variations assumptions writes 
\begin{equation}\label{eq:entro.app.fracture}
\begin{aligned}
\frac{\varrho_f \df}{T_f}\partial_t e_f & + \frac{p_f\varrho_f}{T_f}{} \df \partial_t \frac{1}{\varrho_f}
  + \frac{1}{T_f}\varrho_f\bV_f\cdot\nabla e_f + \frac{p_f}{T_f}\varrho_f\bV_f\cdot\nabla {1 \over \varrho_f} + \frac{1}{\Tref}( \div\q_f - \jump{\q_m}_\n)\\
  & - {1 \over T_f} \sum_{\aa \in \{+,-\}} \gamma_\n^\aa(\varrho_m \bV_m )  \(\gamma^\aa e_m - e_f + p_f ({1 \over \gamma^\aa \varrho_m} - {1 \over \varrho_f} ) \)     = {H_{f} \over T_f} - {h_f \over T_f} G_{f}.
\end{aligned}
\end{equation}

\begin{figure}
  \begin{center}
  \includegraphics[scale=.65]{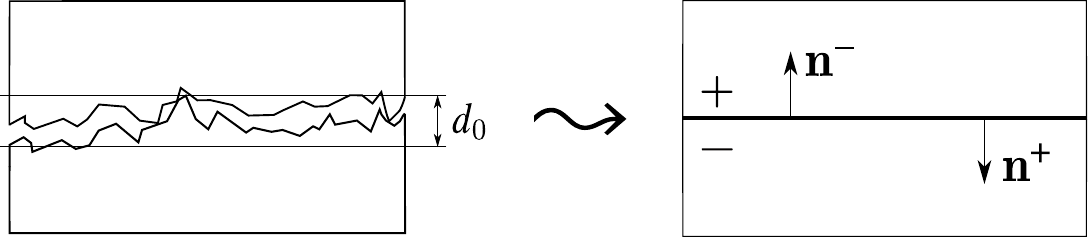}
  \caption{Conceptual fracture model with contact at asperities, $d_0$ being the fracture aperture at contact state.}
  \label{fig:d0}
\end{center}  
\end{figure}

The model is closed by considering no-flow and no-energy flux boundary conditions both for the matrix and fractures. Let us also recall that the pressure and temperature are assumed to be continuous at matrix pressure interfaces. At fracture intersections, the fracture pressure $p_f$ and temperature $T_f$ are classically also continuous, and mass and thermal flux conservation is imposed. 

In the sequel, we use the following terminology to refer to the two mixed-dimensional THM systems: 
\begin{equation}
\mbox{\eqref{eq:matrix}-\eqref{eq:revdef}-\eqref{eq:laws}-\eqref{eq:fracture}: ``\emph{enthalpy-based}'' THM model,}
\tag{$H$-model}
\label{model:enthalpy}
\end{equation}
\begin{equation}
\begin{aligned}
&\mbox{\eqref{eq:matrix}-\eqref{eq:revdef}-\eqref{eq:laws}-\eqref{eq:fracture}
with \eqref{eq:energy.h.matrix} replaced by the approximation \eqref{eq:entro.app.matrix} and}\\
&\mbox{\eqref{eq:energy.h.fracture} replaced by the approximation \eqref{eq:entro.app.fracture}: ``\emph{entropy-based}'' THM model.}
\end{aligned}
\tag{$S$-model}
\label{model:entropy}
\end{equation}

\subsection{Energy estimates}\label{sec:energyestimates} 

Let us first consider the enthalpy-based THM model \eqref{model:enthalpy}. Summing the integral over $\Omega$ of $\eqref{eq:energy.h.matrix} + \partial_t \bu \cdot \eqref{eq:meca}$ and the integral over $\Gamma$ of \eqref{eq:energy.h.fracture}, taking into account the contact and boundary conditions, the following contact persistence condition (which results from \eqref{eq:meca.contact})
$$
\int_\Gamma T_n \jump{\partial_t \bu}_\n d\sigma = 0, 
$$
and the zero dissipation equation \eqref{eq:revdef}, we obtain formally the following energy estimate:
\begin{equation}\label{eq:energyest.h}
\begin{aligned}
  \partial_t \int_\Omega ( E_s + \varrho_m \phi e_m) d\x + \partial_t \int_\Gamma  \varrho_f \df e_f d\sigma 
  + & \int_\Gamma -F T_n |\jump{\partial_t \bu}_\tau| d\sigma  \\ & = \int_\Omega (H_m + \bF \cdot \partial_t \bu)d\x + \int_\Gamma H_f d\sigma.   
\end{aligned}
\end{equation}
Let us now consider the entropy-based THM model \eqref{model:entropy}. By construction, the energy estimate for this model can be deduced from  the previous estimate \eqref{eq:energyest.h} just by adding to the left hand side of \eqref{eq:energyest.h} the term
$$
 \int_\Omega \(-\bV_m \cdot \nabla p + ({T\over \Tref}-1) \div\q_m\)d\x + \int_\Gamma \( -\bV_f \cdot \nabla_\tau p_f + ({T_f \over \Tref}-1) ( \div\q_f  - \jump{\q_m}_\n) \) d\sigma,  
$$
which corresponds to the difference between the exact terms and their approximations. Recalling the definitions \eqref{eq:DarcyFourier.matrix} and \eqref{eq:darcy.fourier.fracture} of $\bV_m$, $\bq_m$, $\bV_f$ and $\bq_f$, this leads to the following stronger form of the energy estimate 
\begin{equation}\label{eq:energyest.s}
\begin{aligned}
  & \partial_t \int_\Omega ( E_s + \varrho_m \phi e_m) d\x + \partial_t \int_\Gamma  \varrho_f \df e_f d\sigma + \int_\Gamma -F T_n |\jump{\partial_t \bu}_\tau| d\sigma \\
  & + \int_\Omega \( {\K \over \eta_m} \nabla p \cdot \nabla p  + {\Lambda_m \over \Tref} |\nabla T|^2 \) d\x 
  + \int_\Gamma \( {C_f \over \eta_f} |\nabla_\tau p_f|^2  +  {\Lambda_f  \over \Tref} |\nabla_\tau T_f|^2\) d\sigma \\ 
 & = \int_\Omega (H_m + \bF \cdot \partial_t \bu)d\x + \int_\Gamma H_f d\sigma.   
\end{aligned}
\end{equation}

\section{Discretisation}\label{sec:discretisation}

The objective is to design a discretisation preserving the energy estimates \eqref{eq:energyest.h} and \eqref{eq:energyest.s} for  respectively the enthalpy-based and entropy-based THM models. Both types of formulations will be compared in the numerical section. The schemes must account for both diffusive and convective dominated energy transport. This motivates the choice of a Finite Volume formulation of the mass and energy equations with possible upwinding of the mobilities. To fix ideas, we will consider in this work the mixed-dimensional Hybrid Finite Volume (HFV) discretisation introduced in \cite{BGGLM16} accounting for the continuity of the pressure and temperature at matrix fracture interfaces.
It will be combined with a mixed variational formulation of the contact-mechanics with a conforming Finite Element discretisation of the displacement field and a facewise constant approximation of the Lagrange multiplier introduced to represent the surface tractions along the fractures. This choice handles fracture networks with corners and intersections, and leads to a local formulation of the contact conditions. It also ensures the discrete persistence property of the contact term and a dissipative frictional term which are key conditions to obtain an energy estimate for the coupled THM systems. 

\subsection{Space and time discretisations}

Let $\cells$ denote the set of polytopal cells, and $\faces$ the set of faces of the mesh, with internal faces gathered in $\faces^{\rm int}$ and boundary faces in $\faces^{\rm ext}$. The subset $\faces_K\subset \faces$ denotes the set of faces of the cell $K\in \cells$. We denote by $\sigma=K|L$ the internal face shared by the two cells $K,L$ gathered in the subset $\cells_\sigma = \{K,L\}$.
The notations $\sigma=K|\cdot$ and $\cells_\sigma = \{K\}$ are used for a face $\sigma\in\faces_K\cap\faces^{\rm ext}$.
The mesh is assumed conforming to the fracture network $\Gamma$ in the sense that there exists a subset $\faces_\Gamma$ of $\faces$ such that
$$
\ol\Gamma = \bigcup_{\sigma \in \faces_\Gamma} \ol\sigma. 
$$
The subset of edges of a face $\sigma$ is denoted by $\edges_\sigma$ and we define the set of edges of $\Gamma$ by $\edges_\Gamma=\bigcup_{\sigma \in \faces_\Gamma} \edges_\sigma$. For a given edge $\zeta\in \edges_\Gamma$, let us denote by $\faces_{\Gamma,\zeta}$ the subset of fracture faces sharing the edge $\zeta$. 
We denote by $|K|$ the $d$-dimensional measure of the cell $K\in\cells$, and by $|\sigma|$ the $(d-1)$-dimensional measure of the face $\sigma\in\faces$.

The HFV method is a Gradient Discretisation (GD) defined by the vector space of discrete unknowns 
\[
\XD=\{v_\D=((v_K)_{K\in\cells},(v_\sigma)_{\sigma\in\faces},(v_\zeta)_{\zeta\in\edges_\Gamma})\}
=\R^{\cells\cup\faces\cup\edges_\Gamma},
\]
and the reconstruction operators $\PiDm,\gradDm$ in the matrix and $\PiDf,\gradDf$ along the fractures (see \cite{BGGLM16} for their detailed definition).
We assume that $\PiDm$ and $\PiDf$ are piecewise constant reconstructions \cite[Definition 2.12]{gdm} and, more specifically, $(\PiDm v_\D)_{|K}=v_K$ for all $K\in\cells$, and $(\PiDf v_\D)_{|\sigma}=v_\sigma$ for all $\sigma\in\faces_\Gamma$.
From the gradient reconstruction operator $\gradDm$ in the matrix, and given any $K\in\cells$ and any symmetric positive definite $d\times d$ matrix $\mathbb{D}_K$, fluxes $F_{K,\sigma}(\mathbb{D}_K;\cdot):X_\D\to\R$ (for $\sigma\in\faces_K$) are defined such that, for all $v_\D,w_\D\in X_\D$ and all $K\in\cells$,
\begin{equation}
\label{eq:def.fluxes.matrix}
\int_K \mathbb{D}_K \gradDm v_\D\cdot\gradDm w_\D d\x =\sum_{\sigma\in\faces_K}F_{K,\sigma}(\mathbb{D}_K;v_\D)(w_K-w_\sigma).
\end{equation}
Likewise, from the fracture tangential gradient reconstruction operator $\gradDf$, for any $\sigma\in\faces_\Gamma$ and any symmetric positive definite $(d-1)\times(d-1)$ matrix $\mathbb{D}_\sigma$, fracture fluxes $F_{\sigma,\zeta}(\mathbb{D}_\sigma;\cdot): \XD\to\R$ are defined (for $\zeta\in\edges_\sigma$) such that, for all $v_\D,w_\D\in X_\D$,
\begin{equation}
\label{eq:def.fluxes.fracture}
\int_\sigma \mathbb{D}_\sigma \gradDf v_\D \cdot\gradDf w_\D d\sigma =\sum_{\zeta\in\edges_\sigma}F_{\sigma,\zeta}(\mathbb{D}_\sigma;v_\D)(w_\sigma-w_\zeta).
\end{equation}
These fluxes are local to each cell $K\in \cells$ (resp.\ each fracture face $\sigma \in \faces_\Gamma$) in the sense that they only depend on $(v_K,(v_\sigma)_{\sigma\in\faces_K})$ (resp. $(v_\sigma,(v_\zeta)_{\zeta\in\edges_\sigma})$) \cite{BGGLM16}. 

In order to account for convective dominated energy transport, an upwind approximation needs to be introduced according to the sign of the Darcy flux $F_{K,\sigma}(\K_K;p_\D)$ using the following notations. 
Let the index ``$K\sigma,+$'' denote either $\sigma$ if no upwinding is used, or an upwind choice between $K$ and $L$ if $\sigma = K|L \in \faces^{\rm int} \setminus \faces_\Gamma$, or an upwind choice between $K$ and $\sigma$ if $\sigma =K|\cdot \in \faces^{\rm ext}\cup \faces_\Gamma$ (note that, if $\sigma=K|L$ is an internal face and not a fracture face, this upwinding must not depend on $K$ or $L$; for example, $\varrho_{K\sigma,+}=\varrho_{L\sigma,+}$).

Similarly, for a given fracture edge $\zeta$,  the index ``$\sigma\zeta,+$'' denote either $\zeta$ if no upwinding is used, or an upwind value between $\sigma$ and $\sigma'$  if $\faces_{\Gamma,\zeta} = \{\sigma,\sigma'\}$ and $\zeta$ is not a boundary edge, or an upwind value between $\sigma$ and $\zeta$ in the other cases. The choice of the upwind value is done according to the sign of the Darcy flux $F_{\sigma,\zeta}(C_{f,\sigma};p_\D)$.

For any fluid thermodynamical property $\Xi = h, e,\eta,\varrho$, and $p_\D, T_\D \in \XD$, we define $\Xi_\D \in \XD$ component by component by setting $\Xi_\nu = \Xi(p_\nu,T_\nu)$ for all $\nu \in \cells\cup\faces\cup\edges_\Gamma$.

The displacement field $\bu$ is discretised using a finite dimensional subspace $\U_\D$ of $\U_0$, typically given by a Finite Element Method.

We consider a time discretisation $(t^n)_{n=0,\ldots,N}$ of the time interval $(0,t_F)$ with $t^0 = 0$ and $t^N= t_F$, and denote by $\dtn = t^{n}-t^{n-1}$ the time step $n$. 
If $f=(f^n)_{n=0,\ldots,N}$ is a family of functions or vectors of $\XD$, the discrete time derivative of $f$ is defined as
$$
\deln f=\frac{f^{n}-f^{n-1}}{\dtn}.
$$
The second time derivative of the discrete displacement field is defined by 
\[
\deln \dot{\bu}_\D = 2\frac{\dot{\bu}_\D^{n}-\dot{\bu}_\D^{n-1}}{t^{n}-t^{n-2}}\quad\mbox{ with }\quad\dot{\bu}_\D^{n}=\deln \bu_\D.
\]
It corresponds to the generalisation to non-uniform time steps of the standard centered discretisation of $\partial_t^2\bu$. 
To alleviate notations, whenever $w_\D\in \XD^{N+1}$ and $\bullet=\{m,f\}$, we set $\delnbullet w_\D = \deln(\PiDbullet w_\D)$. We also consider the discrete time derivative of products of functions and elements of $\XD$, or elements of $\XD$,
by setting: $\delnbullet(f v_\D w_\D)=\deln(f\PiDbullet v_\D\PiDbullet w_\D)$ for all families of functions $f=(f^n)_{n=0,\ldots,N}$ and elements $v_\D,w_\D\in\XD^{N+1}$.

The discrete  porosities $(\phi_\D^n)_{n=0,\ldots,N}$ and skeleton entropies $(S_{s,\D}^n)_{n=0,\ldots,N}$ are families of piecewise constant functions $\Omega\to\R$ on $\cells$ such that $\phi_\D^0,S_{s,\D}^0$ are given (e.g., as projections of the continuous initial porosity and entropy) and, for all $n=1,\ldots,N$,
\begin{align}
\deln\phi_\D ={}&  b \pi_\cells(\deln\div\bu_\D) - \alpha_\phi \delnm T_\D  +\frac1N  \delnm p_\D,
\label{eq:def.discrete.phi}\\
\label{eq:dtSsD}
\deln S_{s,\D}={}& \alpha_s K_s \pi_\cells(\deln\div\bu_\D)- \alpha_\phi \delnm p_\D+\frac{C_s}{\Tref} \delnm T_\D,
\end{align}
where $\pi_\cells$ is the projection on piecewise constant functions on $\cells$ (that is, $(\pi_\cells f)_{|K}=\frac{1}{|K|}\int_K f$ for all $K\in\cells$).

In the fracture, we define the discrete apertures as $(\dfD^n)_{n=0,\ldots,N}$, with $\dfD^n:\Gamma\to\R$ given by
\begin{equation}\label{eq:def.df}
\dfD^n = \pi_{\faces,\Gamma}(d_0 - \jump{\bu_\D^n}_\n).
\end{equation}
Here, $\pi_{\faces,\Gamma}$ is the $L^2$-projection on facewise constant functions on $\Gamma$.

We also define, for $n=1,\ldots,N$, the function $\widehat{G}_{m}^n:\Omega\to\R$
which is piecewise constant on $\cells$ equal on $K\in\cells$ to the average of $G_{m}$ on $(t^{n-1},t^{n})\times K$. The piecewise constant functions  $\widehat{H}_{m}^n$ and $\widehat{\bF}^n$ on $\cells$, as well as 
$\widehat{G}_{f}^n$ and $\widehat{H}_{f}^n$ on $\faces_\Gamma$ are similarly defined. 
In the following, we use the notations
\begin{equation}\label{def:darcy-fluxes}
V^n_{K,\sigma} =  F_{K,\sigma}\({\K(\phi_K^{n-1}) \over \eta^{n-1}_K};p^{n}_\D\) \quad \mbox{  and } \quad 
V^n_{\sigma,\zeta} =  F_{\sigma,\zeta} \({C_f(\dfsig^{n-1}) \over \eta^{n-1}_\sigma};p^{n}_\D\), 
\end{equation}
for the matrix and fracture Darcy flux, as well as
\begin{equation}\label{def:fourier-fluxes}
Q^n_{K,\sigma} =  F_{K,\sigma}\(\Lambda_{m}(\phi_K^{n-1});T^{n}_\D\) \quad \mbox{  and  } \quad 
Q^n_{\sigma,\zeta} =  F_{\sigma,\zeta} \(\Lambda_f(\dfsig^{n-1});T^{n}_\D\)
\end{equation}
for the matrix and fracture Fourier fluxes. 
Note that $\K$, $\Lambda_m$, $C_f$ and  $\Lambda_f$ could also explicitly depend on $\x$. This dependence is omitted for the sake of simplicity. 

\subsection{Discretisation of the mass and energy equations}\label{sec:discretemassenergy}

Let us first consider the HFV discretisation of the mass and energy conservation equations of the enthalpy-based THM model \eqref{model:enthalpy}. For all time step $n=1,\ldots,N$, it couples the discrete mass conservation  equations 
\begin{subequations}
 \label{eq:scheme.mass}
  \begin{equation}  
    |K|\deln (\varrho_K \phi_K)  + \sum_{\sigma\in\faces_K} \varrho_{K\sigma,+}^{n} V^n_{K,\sigma} = |K|\widehat{G}^{n}_{m,K} \qquad\forall K\in\cells,\label{eq:mass.discrete:matrix} 
 \end{equation}   
\begin{equation}  
\begin{aligned}
|\sigma|{}&\deln (\varrho_\sigma \dfsig) + \sum_{\zeta\in\edges_\sigma} \varrho_{\sigma\zeta,+}^{n} V^n_{\sigma,\zeta}
 - \sum_{K\in \cells_\sigma} \varrho^{n}_{K\sigma,+} V^n_{K,\sigma} = |\sigma|\widehat{G}^{n}_{f,\sigma} \qquad\forall \sigma\in\faces_\Gamma,\label{eq:mass.discrete:fracture}
\end{aligned}
\end{equation}
\begin{equation}
   \sum_{K\in \cells_\sigma} V_{K,\sigma}^n  = 0 \qquad\forall \sigma \in\faces \backslash \faces_\Gamma,  
  \label{eq:mass.flux:cons:matrix}
\end{equation}
\begin{equation}
\begin{aligned}  
  &  - \sum_{\sigma \in \faces_{\Gamma,\zeta}}\varrho_{\sigma\zeta,+}^{n} V_{\sigma,\zeta}^n   = 0 {} &\qquad\forall \zeta \in\edges_\Gamma, 
  \label{eq:mass.flux:cons:fracture}
  \end{aligned}
\end{equation}
\end{subequations}
to the discrete total energy conservation equations
\begin{subequations}
 \label{eq:scheme.energy}
 \begin{equation}\label{eq:energy.discrete:matrix}
   \begin{aligned}
    |K|\( T_K^n \deln S_{s,K} & + p^n_K \deln \phi_K +  \deln (\varrho_K \phi_K e_K)\)  \\
    & + \sum_{\sigma\in\faces_K} ( \varrho_{K\sigma,+}^{n} h_{K\sigma,+}^n V^n_{K,\sigma} + Q_{K,\sigma}^n) = |K|\widehat{H}^{n}_{m,K} \qquad\forall K\in\cells,
    \end{aligned}
 \end{equation}   
\begin{equation}  
\begin{aligned}
|\sigma| \( p_\sigma^n \deln \dfsig & + \deln (\varrho_\sigma \dfsig e_\sigma)\)  + \sum_{\zeta\in\edges_\sigma} (\varrho_{\sigma\zeta,+}^{n} h_{\sigma\zeta,+}^{n} V^n_{\sigma,\zeta} + Q^n_{\sigma,\zeta} ) \\
& - \sum_{K\in \cells_\sigma} ( \varrho^{n}_{K\sigma,+} h^{n}_{K\sigma,+} V^n_{K,\sigma} + Q^n_{K,\sigma}) = |\sigma|\widehat{H}^{n}_{f,\sigma} \qquad\forall \sigma \in\faces_\Gamma,\label{eq:energy.discrete:fracture}
\end{aligned}
\end{equation}
\begin{equation}
\sum_{K\in \cells_\sigma} Q_{K,\sigma}^n   = 0 \qquad\forall \sigma  \in\faces \backslash \faces_\Gamma, 
  \label{eq:energy.flux:cons:matrix}
\end{equation}
\begin{equation}
\begin{aligned}  
  &  -\sum_{\sigma \in \faces_{\Gamma,\zeta}} ( \varrho_{\sigma\zeta,+}^{n} h_{\sigma\zeta,+}^{n} V_{\sigma,\zeta}^n + Q_{\sigma,\zeta}^n )  = 0 {} &\qquad\forall \zeta \in\edges_\Gamma, 
  \label{eq:energy.flux:cons:fracture}
  \end{aligned}
\end{equation}
\end{subequations}
  Note that \eqref{eq:mass.flux:cons:matrix}, \eqref{eq:energy.flux:cons:matrix} embed both the conservation of fluxes across internal faces, as well as the zero-flux conditions on boundary faces.
The same remark holds for the flux conservation equations \eqref{eq:mass.flux:cons:fracture}, \eqref{eq:energy.flux:cons:fracture} at fracture edges but keeping the full nonlinear fluxes due to the dependence on $\sigma$ of the upwind density and enthalpy when $\faces_{\Gamma,\zeta}$ contains more than two faces (which happens at internal crossing lines between three or more fractures). 
Similar considerations hold for \eqref{eq:entropy.flux:cons:matrix} and \eqref{eq:entropy.flux:cons:fracture} below.

Temporarily dropping the time index $n$ for simplicity, the HFV discretisation of the approximate entropy equations  \eqref{eq:entro.app.matrix}-\eqref{eq:entro.app.fracture} is based on 
the following discretisation of $\varrho_m \bV_m\cdot\nabla w_m = \div(\varrho_m w_m \bV_m ) - w_m~\div(\varrho_m\bV_m)$ on a given cell $K$, which includes a possible upwinding: for $w_\D\in X_\D$,
\begin{equation*}
\sum_{\sigma\in\faces_K}  w_{K\sigma,+} \varrho_{K\sigma,+} V_{K,\sigma} - w_K\sum_{\sigma\in\faces_K} \varrho_{K\sigma,+} V_{K,\sigma}
=\sum_{\sigma\in\faces_K}  \varrho_{K\sigma,+} V_{K,\sigma}(w_{K\sigma,+} - w_K).
\end{equation*}
This is a key choice which preserves the link between the non conservative entropy and the conservative energy formulations in the sense that it is designed to satisfy the following discrete version of $w_m~\div(\varrho_m\bV_m)+\varrho_m \bV_m\cdot\nabla w_m = \div(\varrho_m w_m \bV_m )$:
$$
w_K  \sum_{\sigma\in\faces_K}  \varrho_{K\sigma,+} V_{K,\sigma} + \sum_{\sigma\in\faces_K} \varrho_{K\sigma,+} V_{K,\sigma}(w_{K\sigma,+} - w_K) =
\sum_{\sigma\in\faces_K}  w_{K\sigma,+} \varrho_{K\sigma,+} V_{K,\sigma}.  
$$
Likewise, in the fractures, we use the following discretisations of $\varrho_f \bV_f\cdot\nabla_\tau w_f$  
$$
\sum_{\zeta\in\edges_\sigma}\varrho_{\sigma\zeta,+}
V_{\sigma,\zeta} (w_{\sigma\zeta,+} -w_\sigma), 
$$
on $\sigma\in \faces_\Gamma$, and 
$$
 \sum_{\sigma \in \faces_{\Gamma,\zeta}} - \varrho_{\sigma\zeta,+}
V_{\sigma,\zeta} (w_{\sigma\zeta,+} - w_\zeta), 
$$
on $\zeta\in \edges_\Gamma$. In the same spirit, the terms such as $\varrho_m \phi \partial_t w_m$ in the matrix are discretised by
$\PiDm \varrho^{n-1} \phi_\D^{n-1} \delnm w_\D$ to ensure that
\begin{equation}\label{eq:disc.product.rule}
\PiDm \varrho_\D^{n-1} \phi_\D^{n-1} \delnm w_\D + \PiDm w_\D^n \times \delnm (\varrho_\D \phi_\D)
= \delnm ( \varrho_\D \phi_\D w_\D ).  
\end{equation}
The same discretisation is applied for the terms such as $\varrho_f \df \partial_t w_f$ in the fractures. 
Following this methodology, the approximate entropy equations are discretised for all time step $n=1,\ldots,N$ by 
\begin{subequations}
 \label{eq:scheme.entropy}
 \begin{equation}\label{eq:entropy.discrete:matrix}
   \begin{aligned}
    |K| \deln S_{s,K} & +    {|K| \over T_K^n} \varrho^{n-1}_K \phi^{n-1}_K \(  \deln e_K+   p_K^n \deln {1 \over \varrho_K} \)  +  {1 \over \Tref}  \sum_{\sigma\in\faces_K}  Q_{K,\sigma}^n\\
  & +  {1 \over T_K^n} \sum_{\sigma\in\faces_K} \varrho_{K\sigma,+}^{n} V^n_{K,\sigma} \( e_{K\sigma,+}^n  - e^n_K + p_K^n ( {1 \over \varrho_{K\sigma,+}^n} - {1 \over \varrho^n_K})\)\\
  &   = {|K| \over T_K^n} ( \widehat{H}^{n}_{m,K} - h_K^n \widehat{G}^{n}_{m,K}) \qquad\forall K\in\cells,
    \end{aligned}
 \end{equation}   
\begin{equation}  
\begin{aligned}
  {|\sigma| \over T_\sigma^n}  \varrho^{n-1}_\sigma \dfsig^{n-1} & \( \deln  e_\sigma + p_\sigma^n \deln  {1 \over \varrho_\sigma}  \) + {1 \over \Tref} \( \sum_{\zeta\in\edges_\sigma} Q^n_{\sigma,\zeta} - \sum_{K\in \cells_\sigma} Q^n_{K,\sigma} \) \\
  & + {1 \over T_\sigma^n} \sum_{\zeta\in\edges_\sigma} \varrho_{\sigma\zeta,+}^{n}  V^n_{\sigma,\zeta} \( e_{\sigma\zeta,+}^{n} - e_\sigma^n + p_\sigma^n ( {1 \over \varrho^n_{\sigma\zeta,+}} - {1 \over \varrho_\sigma^n}  ) \)    \\
  & - {1 \over T_\sigma^n}\sum_{K\in \cells_\sigma} \varrho^{n}_{K\sigma,+} V^n_{K,\sigma} \( e^{n}_{K\sigma,+} - e_\sigma^n + p_\sigma^n ({1 \over \varrho^{n}_{K\sigma,+}} - {1 \over \varrho_\sigma^n}) \)\\
 &   = {|\sigma| \over T_\sigma^n} ( \widehat{G}^{n}_{f,\sigma} - h_\sigma^n \widehat{G}^{n}_{f,\sigma} ) \qquad\forall \sigma\in\faces_\Gamma,\label{eq:entropy.discrete:fracture}
\end{aligned}
\end{equation}
\begin{equation}
\sum_{K\in \cells_\sigma} Q_{K,\sigma}^n   = 0 \qquad\forall \sigma \in \faces\backslash \faces_\Gamma,
  \label{eq:entropy.flux:cons:matrix}
\end{equation}
\begin{equation}
\begin{aligned}  
  &  - {1 \over T_\zeta^n} \sum_{\sigma \in \faces_{\Gamma,\zeta}}\varrho_{\sigma\zeta,+}^{n}  V_{\sigma,\zeta}^n \( e_{\sigma\zeta,+}^{n} - e_\zeta^n + p_\zeta^n({1 \over \varrho_{\sigma\zeta,+}^n} - {1 \over \varrho_{\zeta}^n}  ) \) - {1 \over \Tref} \sum_{\sigma \in \faces_{\Gamma,\zeta}} Q_{\sigma,\zeta}^n   = 0 {} &\quad\forall \zeta \in\edges_\Gamma, 
  \label{eq:entropy.flux:cons:fracture}
  \end{aligned}
\end{equation}
\end{subequations}

By construction, applying this space time discretisation to the entropy equations \eqref{eq:entro.matrix}-\eqref{eq:entro.fracture}
would lead to an equivalence between the entropy and energy formulations at the discrete level when combined with the discrete mass equations. It results that the following Lemma states the equivalence up to the terms which have been neglected or linearised in the approximate entropy equations. 

\begin{lemma}\label{lemma_link_entropy_enthalpy_discrete}
  For all cell $K\in \cells$, $T^n_K \times \eqref{eq:entropy.discrete:matrix} + h_K^n \times\eqref{eq:mass.discrete:matrix}$ is equivalent to \eqref{eq:energy.discrete:matrix} upon correcting the left-hand side by adding
  $$
  \left({T_K^n \over \Tref} - 1\right) \sum_{\sigma\in\faces_K}  Q_{K,\sigma}^n - \sum_{\sigma\in\faces_K} V^n_{K,\sigma} ( p_{K\sigma,+}^{n} - p_K^n). 
  $$
  For all fracture face $\sigma \in \faces_\Gamma$, $T^n_\sigma \times \eqref{eq:entropy.discrete:fracture} + h_\sigma^n \times\eqref{eq:mass.discrete:fracture}$  is equivalent to \eqref{eq:energy.discrete:fracture} upon correcting the left-hand side by adding
  $$
  \left({T_\sigma^n \over \Tref} - 1\right) \( \sum_{\zeta\in\edges_\sigma}  Q_{\sigma,\zeta}^n - \sum_{K \in \cells_\sigma}  Q_{K,\sigma}^n \)
  - \sum_{\zeta\in\edges_\sigma} V^n_{\sigma,\zeta} ( p_{\sigma\zeta,+}^{n} - p_\sigma^n) + \sum_{K\in \cells_\sigma} V^n_{K,\sigma} ( p^{n}_{K\sigma,+} - p_\sigma^n). 
  $$
For all fracture edge $\zeta \in \edges_\Gamma$, $T^n_\zeta \times\eqref{eq:entropy.flux:cons:fracture} + h_\zeta^n \times \eqref{eq:mass.flux:cons:fracture}$ is equivalent to \eqref{eq:energy.flux:cons:fracture} upon correcting the left-hand side by adding
 $$
  - \left({T_\zeta^n \over \Tref} - 1\right) \sum_{\sigma \in \faces_{\Gamma,\zeta}} Q_{\sigma,\zeta}^n  + \sum_{\sigma \in \faces_{\Gamma,\zeta}} V_{\sigma,\zeta}^n ( p_{\sigma\zeta,+}^{n} - p_\zeta^n). 
  $$
  
\end{lemma}

\begin{proof}
  Let us detail only the proof for the cell terms since fracture face and edge terms are similar.
  Gathering the terms of $T^n_K \times  \eqref{eq:entropy.discrete:matrix} + h_K^n \times \eqref{eq:mass.discrete:matrix}$  leads to the equation:
  \begin{equation*}
    \begin{aligned}
      &  |K| \left[ T_K^n \deln S_{s,K}  +  h_K^n \deln (\varrho_K \phi_K)  + \varrho^{n-1}_K \phi^{n-1}_K \(  \deln e_K+   p_K^n \deln {1 \over \varrho_K} \) \right]\\
     & + h_K^n \sum_{\sigma\in\faces_K} \varrho_{K\sigma,+}^{n} V^n_{K,\sigma} +  \sum_{\sigma\in\faces_K} \varrho_{K\sigma,+}^{n} V^n_{K,\sigma} \( e_{K\sigma,+}^n  - e^n_K + p_K^n ( {1 \over \varrho_{K\sigma,+}^n} - {1 \over \varrho^n_K})\) \\
   &      +     {T_K^n \over \Tref}  \sum_{\sigma\in\faces_K}  Q_{K,\sigma}^n = |K| \widehat{H}^{n}_{m,K}. 
   \end{aligned}    
  \end{equation*}
  Recalling that $h=e+p/\varrho$, we obtain
  \begin{equation*}
    \begin{aligned}
      &  |K| \left[ T_K^n \deln S_{s,K}  +  \( e_K^n \deln (\varrho_K \phi_K)  + \varrho^{n-1}_K \phi^{n-1}_K \deln e_K \) \right.\\ 
      &\qquad\qquad\qquad\qquad  + \left. p_K^n \( \varrho^{n-1}_K \phi^{n-1}_K  \deln {1 \over \varrho_K} + {1 \over \varrho_K^n} \deln (\varrho_K \phi_K) \) \right]\\
     &   \sum_{\sigma\in\faces_K} \varrho_{K\sigma,+}^{n} V^n_{K,\sigma} \( h_{K\sigma,+}^n   + ( p_K^n - p_{K\sigma,+}^n ) {1 \over \varrho_{K\sigma,+}^n} \) 
         +     {T_K^n \over \Tref}  \sum_{\sigma\in\faces_K}  Q_{K,\sigma}^n = |K| \widehat{H}^{n}_{m,K}. 
   \end{aligned}    
  \end{equation*}

  Invoking \eqref{eq:disc.product.rule}, this relation reduces to 
  \begin{equation*}
    \begin{aligned}
      &  |K| \( T_K^n \deln S_{s,K}  +  p_K^n \deln \phi_K  +  \deln ( \varrho_K \phi_K e_K ) \) 
       +  \sum_{\sigma\in\faces_K} \( \varrho_{K\sigma,+}^{n} h_{K\sigma,+}^{n}  V^n_{K,\sigma} +  Q_{K,\sigma}^n\) \\
     & - \sum_{\sigma\in\faces_K}  V^n_{K,\sigma} ( p_{K\sigma,+}^n  - p^n_K)  + \left({T_K^n \over \Tref} -1\right) \sum_{\sigma\in\faces_K} Q_{K,\sigma}^n
    = |K| \widehat{H}^{n}_{m,K},  
   \end{aligned}    
  \end{equation*}  
  which is the expected result for the cell term. 
\end{proof}

\subsection{Mixed variational discretisation of the contact-mechanical model}\label{sec:discretecontactmechanics}

The subspace $M_\D\subset L^2(\Gamma)$ denotes the set of piecewise constant functions on the partition $\faces_\Gamma$ and we set $\M_\D = (M_\D)^d$. For $\l_\D$ in $\M_\D$,  we use the decomposition $\l_\D = (\lambda_{\D,\n},\l_{\D,\tau})$ with $\lambda_{\D,\n} = \l_\D\cdot {\bf n}^+$, $\l_{\D,\tau} = \l_\D - \lambda_{\D,\n} {\bf n}^+$. We denote by $\l_\sigma$ the constant value of $\l_\D \in \M_\D$ on the face $\sigma\in \faces_\Gamma$ and by $\lambda_{n,\sigma}$ and $\l_{\tau,\sigma}$ its normal and tangential components. 

We consider the mixed formulation based on facewise constant Lagrange multipliers, which seamlessly deals with fracture intersections, corners and tips and leads to local expressions of the discrete contact conditions and efficient semi-smooth Newton solvers. On the other hand, this approach requires to assume the following uniform inf-sup condition between the space $\U_\D$ of displacement fields and the space $\M_\D$ of Lagrange multipliers: there exists $c_\star$ independent on the mesh such that 
\begin{equation}
  \label{eq_infsupUhMh}
  \inf_{\m_\D \in \M_\D} \sup_{\bv_\D \in \U_\D}  { \dsp \int_\Gamma \m_\D \cdot \jump{\bv_\D} d\sigma \over \|\bv_\D\|_{\U_0} \|\m_\D\|_{H^{-{1\over 2}}(\Gamma)^d}}\geq c_\star > 0. 
\end{equation}

Let us define the discrete dual cone of normal Lagrange multipliers as
$$
\Lambda_\D = \{\lambda_{\D,\n}\in M_\D \,|\, \lambda_{\D,\n} \geq 0 \mbox{ on } \Gamma\},  
$$
and the discrete dual cone of vectorial Lagrange multipliers given $\lambda_{\D,\n} \in \Lambda_\D$ as
$$
\L_\D(\lambda_{\D,\n}) = \{ \m_\D = (\mu_{\D,\n},\m_{\D,\tau}) \in \M_\D \,|\, \mu_{\D,\n} \geq 0, |\m_{\D,\tau}| \leq F \lambda_{\D,\n}  \mbox{ on } \Gamma\}.   
$$
Note that the friction coefficient $F$ is assumed to be facewise constant on the partition $\faces_\Gamma$. 

The mixed discretisation of the quasi static contact-mechanical model reads: find $(\bu_\D,\l_\D)\in \U_\D\times \L_\D(\lambda_{\D,\n})$  such that, for all $(\bv_\D,\m_\D)\in \U_\D\times \L_\D(\lambda_{\D,\n})$,
\begin{subequations}\label{Mixed_meca}
  \begin{equation}\label{Mixed_meca_eq}
    \begin{aligned} 
 & \dsp \int_\Omega \( m_0 \deln \dot{\bu}_\D \cdot \bv_\D +  \bbsig( \bu^n_\D): \bbeps( \bv_\D) - \left[ b ~ \PiDm p^n_\D  + \alpha_s K_s (\PiDm T_\D^n - \Tref) \right] ~\div( \bv_\D) \) \d\x\\
& \quad \quad +  \int_\Gamma \l^n_\D  \cdot \jump{\bv_\D} d\sigma 
+ \int_\Gamma \PiDf p^n_\D ~\jump{ \bv_\D}_\n  ~\d\sigma 
\dsp =  \int_\Omega \widehat{\bF}^n \cdot \bv_\D ~\d\x,
\end{aligned}
    \end{equation}
    \begin{equation}\label{Mixed_meca_ineq}
 \dsp \int_\Gamma \( (\mu_{\D,\n}-\lambda^n_{\D,\n})  \jump{\bu^n_\D}_\n + (\m_{\D,\tau}-\l^n_{\D,\tau}) \cdot \jump{\deln {\bu}_\D}_\tau \) d\sigma \leq 0, 
\end{equation}
\end{subequations}
The variational inequality in \eqref{Mixed_meca_ineq} is equivalent to the contact conditions  between $\l_\sigma$ and the face average of $\jump{\bu_\D}$ denoted by $\jump{\bu_\D}_\sigma$  (see e.g.~Lemma 4.1 of \cite{BDMP:21}). This is a key property to obtain the following discrete persistence property of the normal contact term
\begin{equation}
  \label{persistency_discrete}
  \int_\Gamma \lambda^n_{\D,\n}   \jump{\deln \bu_\D}_\n d\sigma \geq 0, 
\end{equation}
leading to the dissipative property of the contact term 
\begin{equation}
  \label{dissipative_contact_discrete}
  \int_\Gamma \l^n_{\D} \cdot  \jump{\deln \bu_\D} d\sigma \geq  \int_\Gamma F \lambda^n_{\D,\n} |\jump{\deln \bu_\D}_{\tau} | d\sigma \geq 0.  
\end{equation}
We refer the interested reader to \cite{BDMP:21} for the proof.

\subsection{Summary of the schemes and discrete energy estimates}\label{sec:discreteenergyestimates} 

To summarise, the discretisations of the enthalpy-based and entropy-based formulations are as follows:
\begin{equation}
\mbox{\eqref{eq:scheme.mass}-\eqref{eq:scheme.energy}-\eqref{Mixed_meca}-\eqref{eq:def.discrete.phi}-\eqref{eq:dtSsD}-\eqref{eq:def.df} for the enthalpy-based model \eqref{model:enthalpy}.}
\tag{$H$-scheme}
\label{scheme:enthalpy}
\end{equation}
\begin{equation}
\mbox{\eqref{eq:scheme.mass}-\eqref{eq:scheme.entropy}-\eqref{Mixed_meca}-\eqref{eq:def.discrete.phi}-\eqref{eq:dtSsD}-\eqref{eq:def.df} for the entropy-based model \eqref{model:entropy}.}
\tag{$S$-scheme}
\label{scheme:entropy}
\end{equation}

Let us  define the discrete skeleton internal energy $\Ens^n:\Omega\to\R$ by 
\begin{equation*}
\Ens^n = \frac12 \begin{bmatrix}\PiDm p^n_\D & \PiDm T^n_\D\end{bmatrix}M\begin{bmatrix}\PiDm p^n_\D \\ \PiDm T^n_\D\end{bmatrix} + \alpha_s K_s \Tref ~\div \bu^n_\D 
 + \mu |\bbeps(\bu_\D^n)|^2+ {\lambda \over 2} (\div\bu_\D^n)^2.
\end{equation*}

Then, the following propositions state the discrete energy estimates for both formulations of the energy equation. 

\begin{proposition}\label{prop_energy_est_energy}
The scheme \eqref{scheme:enthalpy} based on the conservative enthalpy formulation of the energy equation satisfies, for all $n=1,\ldots,N$, the following energy estimate 
  \begin{equation}\label{energy_est_energy}
\begin{aligned}
   \int_\Omega {}& {m_0\over 2} \deln |\dot{\bu}_\D|^2 d\x + \int_\Omega \deln \Ens +  \int_\Omega \delnm(\varrho_\D\phi_\D e_\D) + \int_\Gamma \delnf(\varrho_\D\dfD e_\D)  \\
& + \int_\Gamma F \lambda^n_{\D,\n} |\jump{\deln \bu_\D}_{\tau} | d\sigma  \le{}  
\int_\Omega \( \widehat{H}^{n}_{m} + \widehat{\bF}^n \cdot \deln \bu_\D \) d\x 
+ \int_\Gamma \widehat{H}^{n}_{f} d\sigma. 
\end{aligned}
\end{equation}
\end{proposition}  
\begin{proof}
Adding \eqref{Mixed_meca_eq} with $\bv_\D = \deln \bu_\D = \dot{\bu}^n_\D$ to the sums over $\cells$ of \eqref{eq:energy.discrete:matrix}, over $\faces_\Gamma$ of  \eqref{eq:energy.discrete:fracture} and over $\edges_\Gamma$ of \eqref{eq:energy.flux:cons:fracture}, and taking into account the flux conservativity and the homogeneous Neumann boundary conditions  \eqref{eq:mass.flux:cons:matrix}-\eqref{eq:energy.flux:cons:matrix}, together with the definitions of the discrete fracture aperture \eqref{eq:def.df}, porosity \eqref{eq:def.discrete.phi} and skeleton entropy \eqref{eq:dtSsD}, we obtain the following equality 
\begin{equation*}
  \begin{aligned}
  &   \int_\Omega {} m_0 \deln \dot{\bu}_\D \cdot \dot{\bu}^n_\D ~d\x  + \int_\Omega \delnm ( \varrho_\D \phi_\D e_\D)  d\x + \int_\Gamma  \delnf ( \varrho_\D \dfD e_\D) d\sigma \\ 
  &  + \int_\Omega  \( \begin{bmatrix}\PiDm p^n_\D & \PiDm T^n_\D \end{bmatrix}M \delnm \begin{bmatrix} p_\D \\ T_\D\end{bmatrix}d\x    +   \bbsig^e(\bu^{n}_\D):\deln \bbeps(\bu_\D) d\x
    + \alpha_s K_s \Tref ~\div \deln \bu_\D \) d\x \\
  &  + \int_\Gamma \l^n_{\D} \cdot \jump{\deln \bu_\D}  d\sigma  = \int_\Omega \( \widehat{H}^{n}_{m} + \widehat{\bF}^n \cdot \deln \bu_\D \) d\x 
+ \int_\Gamma \widehat{H}^{n}_{f} d\sigma. 
\end{aligned}
\end{equation*}
Then, using the diffusive properties of the implicit Euler time stepping (that is e.g., $\deln \dot{\bu}_\D \cdot \dot{\bu}^n_\D\ge \frac12 |\dot{\bu}^{n}_\D|^2-\frac12 |\dot{\bu}^{n-1}_\D|^2$) and the dissipative property of the discrete contact term  \eqref{dissipative_contact_discrete}, we obtain the energy estimate \eqref{energy_est_energy}. 
\end{proof}

\begin{proposition}\label{prop_energy_est_entropy}
The scheme \eqref{scheme:entropy} based on the non-conservative approximate entropy equation satisfies, for all $n=1,\ldots,N$, the following energy estimate 
  \begin{equation}\label{energy_est_entropy}
\begin{aligned}
   \int_\Omega {}& {m_0\over 2}\deln |\dot{\bu}_\D|^2 d\x + \int_\Omega \deln \Ens +  \int_\Omega \delnm(\varrho_\D\phi_\D e_\D) + \int_\Gamma \delnf(\varrho_\D\dfD e_\D)  \\
& + \int_\Omega \( {\K( \phi_\D^{n-1}) \over \PiDm \eta_\D^{n-1}} \gradDm p_\D^n \cdot \gradDm p_\D^n  + {\Lambda_m(\phi_\D^{n-1}) \over \Tref} |\gradDm T_\D^n|^2 \) d\x \\
 & + \int_\Gamma \( { C_f(\dfD^{n-1})  \over \PiDf \eta_\D^{n-1}} |\gradDf p_\D^n|^2  +  {\Lambda_f(\dfD^{n-1})  \over \Tref} |\gradDf T_\D^n|^2\) d\sigma \\        
& + \int_\Gamma F \lambda^n_{\D,\n} |\jump{\deln \bu_\D}_{\tau} | d\sigma  \le{}  
\int_\Omega \( \widehat{H}^{n}_{m} + \widehat{\bF}^n \cdot \deln \bu_\D \) d\x 
+ \int_\Gamma \widehat{H}^{n}_{f} d\sigma. 
\end{aligned}
\end{equation}
\end{proposition}  
\begin{proof}
  The estimate \eqref{energy_est_entropy} is a consequence of the proof of Proposition \ref{prop_energy_est_energy} combined with Lemma \ref{lemma_link_entropy_enthalpy_discrete} and the definitions of the Darcy \eqref{def:darcy-fluxes} and Fourier \eqref{def:fourier-fluxes}  fluxes from the coercive matrix \eqref{eq:def.fluxes.matrix} and fracture \eqref{eq:def.fluxes.fracture} fluxes. 
\end{proof}

\section{Numerical experiments}\label{sec:numerics} 

Two test cases are considered in this Section. The first is based on a manufactured solution for both the entropy and enthalpy-based models on a square domain with no fracture. The objective is to assess and compare the convergence of the schemes for both models and both for centred and upwind approximations of the thermal convection. Two values of the permeability are tested to induce either convection dominated or equilibrated diffusion and convection thermal regimes. 

The second example considers a Discrete Fracture Matrix (DFM) model introduced in \cite{contact-norvegiens} including a six fracture network. The setting of the simulation follows the test case proposed in \cite{STEFANSSON2021114122} and the physical and numerical behavior of the discrete models are investigated both for the case of a slightly compressible liquid and for the case of a perfect gas.

In all these numerical experiments, the HFV discretisations of the non-isothermal flow are combined with the $\mathbb{P}_2$ conforming Finite Element discretisation of the displacement field. This choice ensures
the inf-sup condition \eqref{eq_infsupUhMh} for the $ \mathbb{P}_2-\mathbb{P}_0 $ mixed formulation of the contact-mechanics \eqref{Mixed_meca}.
It also satisfies the inf-sup condition between the displacement and pressure discrete spaces which prevents potential oscillations of the pressure field at short times in the undrained regime.

The coupled nonlinear system is solved at each time step using a fixed-point method on the function:
$$
\mathbf{g}_{p,T}: (p,T) \underset{\substack{\text { Contact Mechanics } \\ \text { Solve }}}{\longrightarrow} \mathbf{u}_h \underset{\substack{\text { Flow } \\ \text { Solve }}}{\longrightarrow} (\tilde{p},\tilde{T}),
$$
accelerated by a Newton--Krylov algorithm -- which have proven to be efficient within this context in the isothermal case \cite{BDMP:21}. The stopping criteria is set to $10^{-10}$ on the relative residual. We refer to \cite{BRUN20201964} for an investigation of various fixed-point algorithms for THM models. 

At each iteration of the Newton--Krylov algorithm, the $p,T$ sub-system is solved using a Newton--Raphson algorithm and the contact-mechanics is solved using a semi-smooth Newton method. In both cases, the stopping criteria is defined by a relative norm of the residual set to $10^{-10}$ or a scaled maximum Newton increment of $10^{-10}$. 
At each Newton iteration, the linear system is solved using the sparse direct solver SuperLU version 4.3 both for the non-isothermal flow and for the contact-mechanics.

\subsection{Manufactured solution without fractures}\label{sec:test1} 
 We investigate in this section the numerical convergence of the schemes \eqref{scheme:enthalpy} and \eqref{scheme:entropy} for respectively the \eqref{model:enthalpy} and \eqref{model:entropy} models using the following analytical solution 
$$
\begin{aligned}
& \mathbf{u}(\mathbf{x}, t)=10^{-1} \mathrm{e}^{-t}\left(\begin{array}{c}
x^2 y^2 \\
-x^2 y^2
\end{array}\right), \\
& p(\mathbf{x}, t)=\mathrm{e}^{-t} \sin \left(x\right) \sin \left(y\right), \quad T(\mathbf{x}, t)=\mathrm{e}^{-t}\left( 2 -\cos \left(x\right) \cos \left(y\right)\right),  
\end{aligned}
$$
on the domain $\Omega= (0,1)^2\,\mathrm{m}^2$ and time interval $(0,t_F)$ with $t_F =1 \, \mathrm{s}$. The fluid is assumed incompressible and its specific internal energy is defined by $e(T)=T$. 
Two values $\mathbb{K}= 100~\mathbb{I} \, \mathrm{m}^2$ and $\mathbb{K}= \mathbb{I} \, \mathrm{m}^2$ (with $\mathbb{I}$ the identity matrix) of the homogeneous isotropic permeability are considered. They are chosen such that the resulting Darcy velocity $\bV_m$ corresponds to a thermal convection dominated regime in the first case, and to similar orders of magnitude for thermal convection and diffusion in the second case. 
Dirichlet boundary conditions are imposed for $p$, $T$ and $\bu$ on $(0,t_F)\times \partial \Omega$  and the source terms $G_m$, $H_m$ and $\mathbf{F}$ are computed from the analytical solution based on the data set defined in Table \ref{TB1}.
The domain $\Omega$ is discretised using the first family of triangular meshes from \cite{fvca5bench} as illustrated in Figure \ref{Fg:Domaine}. Each mesh indexed by $m\in \{1,2,3,4\}$ includes $\# \cells = 56 \times 4^{m-1}$ triangles.  We consider a uniform time stepping of $(0,t_F)$ with time step $\Delta t = 2. 10^{-5} \, \mathrm{s}$ chosen small enough to reduce the error due to the time discretisation and focus on the convergence in space.

The convergence of the $L^2$ space time errors for $p,T,\mathbf{u}$ and their gradients are exhibited for both schemes and both permeabilities in Figures \ref{fig:convergenceSansFracS-based}-\ref{fig:convergenceSansFracH-based} as functions of the mesh step.
Both the upwind and centred schemes are considered for the thermal convection. Note that the centred schemes are not presented for $\mathbb{K}= 100~\mathbb{I}$ since they fail to provide a solution due to their instability in the convection dominated regime. 
From Figures \ref{fig:convergenceSansFracS-based}-\ref{fig:convergenceSansFracH-based}, we first observe that the discretisations of both models provide a very similar convergence behavior for a given choice of the permeability and of the centred or upwind approximation of the thermal convection terms. 
Regarding the displacement field, second and first order convergence rates are observed in all cases for respectively $\mathbf{u}$ and $\nabla \mathbf{u}$. This is in accordance with the cellwise constant reconstruction $\PiDm$  of the pressure and temperature in the displacement field variational formulation \eqref{Mixed_meca_eq}.
The convergence rates for $p$ and $\nabla p$ are respectively roughly $2$ and $1$ in all cases as could be expected.
On the other hand, the converge rates for $T$ and $\nabla T$ depend on the approximation of the convection term and on the convection--diffusion regime. For equilibrated convection and diffusion, the centred scheme provides a higher convergence rate of order roughly $2$ for $T$ than the upwind scheme of order between $1$ and $2$. An order $1$ is observed on $\nabla T$ for both the centred and upwind schemes. In the convection dominated regime, the upwind scheme exhibits a convergence rate slightly better than $1$ for $T$, and an order between $0.5$ and $1$ for $\nabla T$ slightly better for the entropy-based than for the enthalpy-based model.

\begin{figure}[H]
\centering
{\includegraphics[keepaspectratio=true,scale=.2]{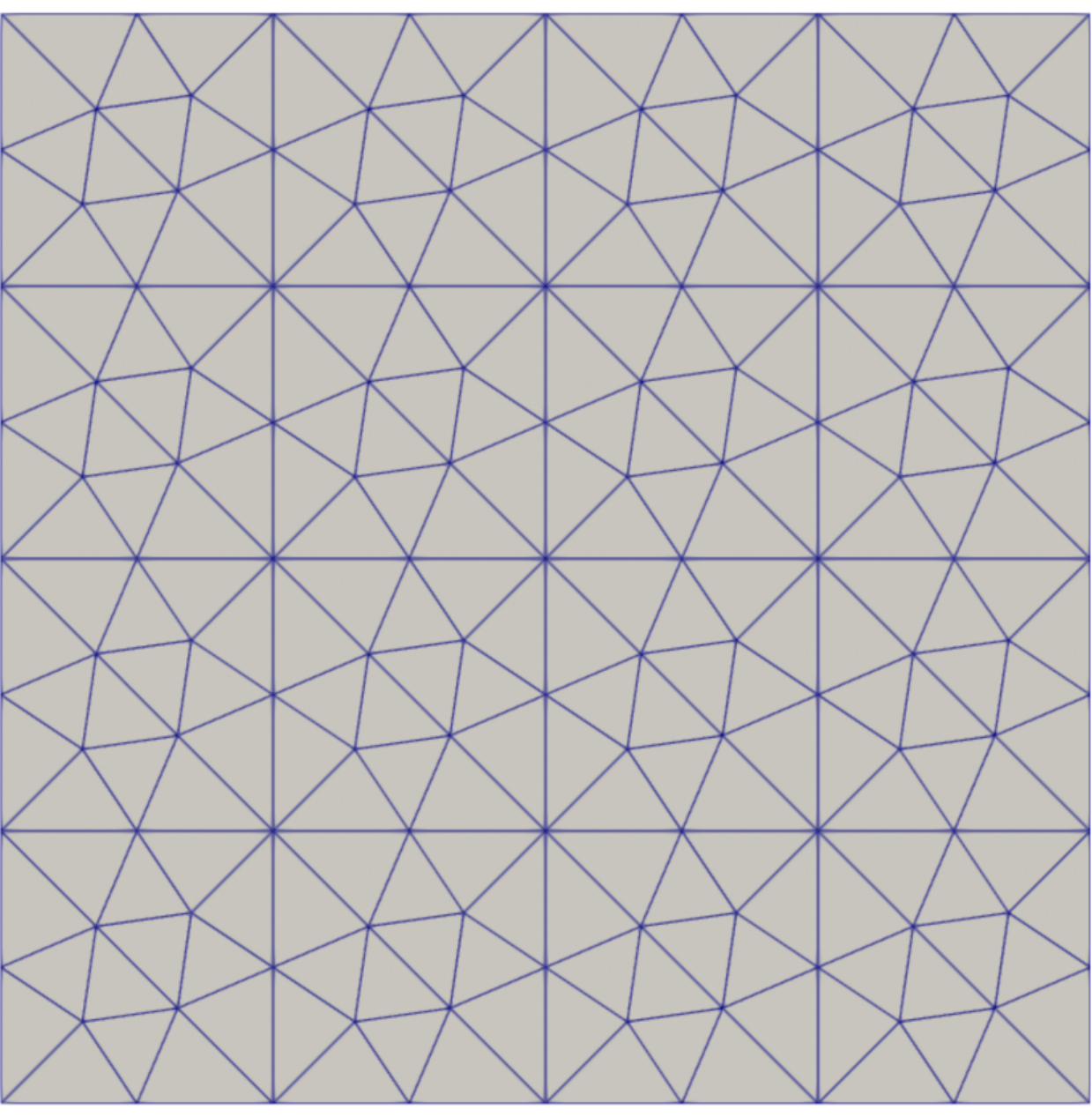}
\caption{Square domain $\Omega$ with its triangular mesh $m=2$ using $56 \times 4 $ cells.}
\label{Fg:Domaine}}
\end{figure}

\begin{table}[H]
 \caption{Material Properties}
\centering
{
 \resizebox{12cm}{!}{
$\begin{array}{llll}
\hline \text { Symbol } & \text { Quantity } & \text { Value } & \text { Unit } \\
\hline E & \text { Young modulus } & 2.5  & \mathrm{~Pa} \\
\nu & \text { Poisson coefficient } & 0.25 & - \\
N & \text { Biot modulus } & 0.25 & \mathrm{~Pa}^{-1} \\
b & \text { Biot coefficient } & 1.0 & - \\
K_s & \text { Bulk modulus } & 2.0 & \mathrm{~Pa} \\
\eta & \text { Fluid viscosity } & 1.0  & \mathrm{~Pa} \mathrm{~s} \\
\phi^0 & \text { Initial porosity} &  4  & -   \\
\Lambda_m & \text { Effective thermal conductivity } & 0.1 & \mathrm{W} \mathrm{~m}^{-1} \mathrm{~K}^{-1} \\
\varrho & \text { The fluid specific density } & 1 & \mathrm{~Kg} \mathrm{~m}^{-3} \\
\alpha_s & \text { The volumetric skeleton thermal dilation coefficient } & 1 &  \mathrm{~K}^{-1} \\
\alpha_\phi & \text { The volumetric thermal dilation coefficient related to the porosity } & 1 & \mathrm{~K}^{-1} \\
\Tref & \text { Reference temperature } & 1 & \mathrm{~K} \\
m_0 & \text { Average fluid skeleton specific density} & 0 & \mathrm{~Kg~m^{-3}} \\
C_s & \text { The skeleton volumetric heat capacity } & 0.5 & \mathrm{~J} \mathrm{~m}^{-3} \mathrm{~K}^{-1} \\
\hline
\end{array}$
}
}
\label{TB1}
\end{table}

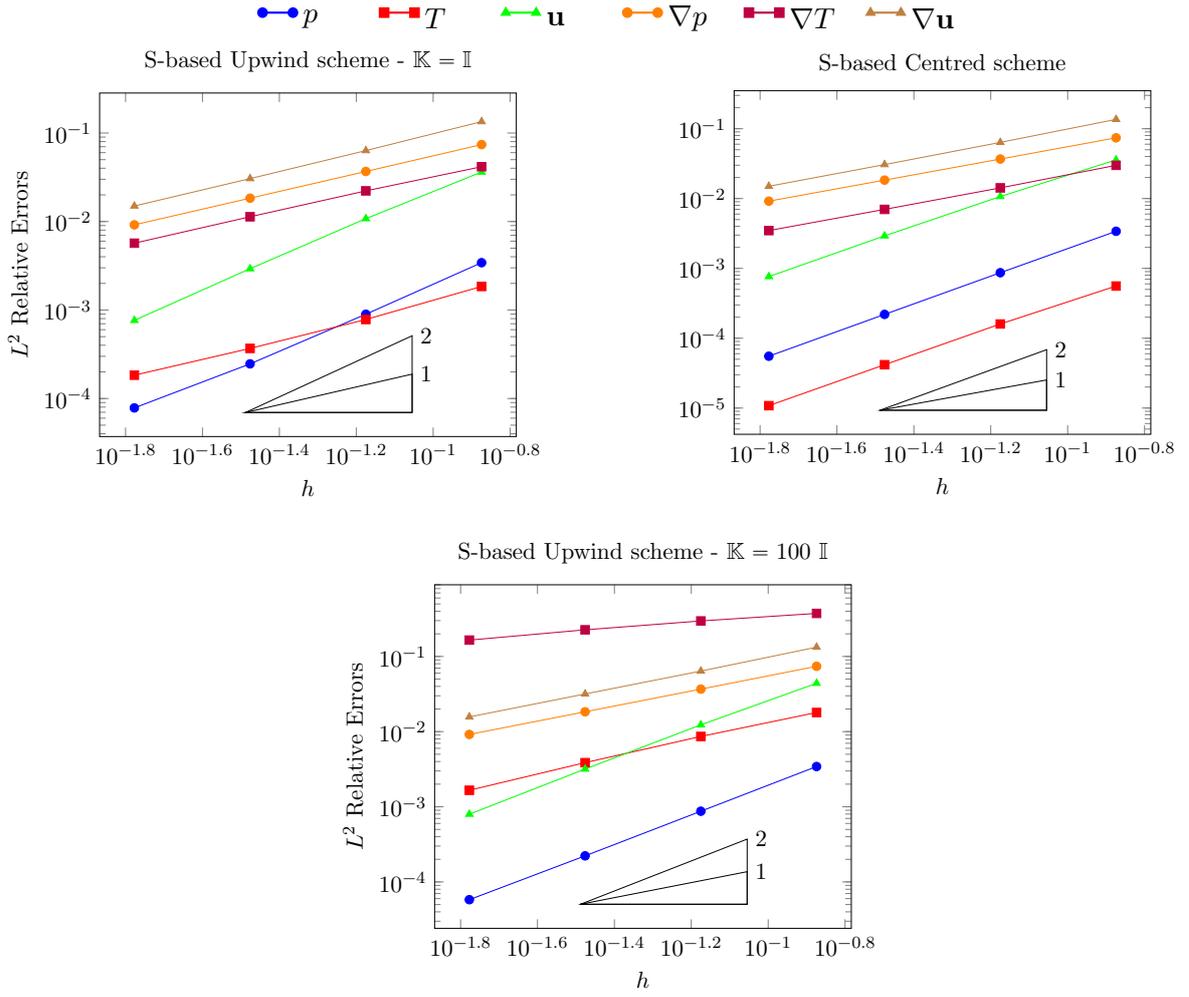
\begin{figure}[H]
    \centering
    
    \begin{tikzpicture}[scale=0.8]
        \node at (0,0) { };
        
        \draw[blue, thick, mark=*, mark options={blue}] plot coordinates {(1,0.1) (1.5,0.1)}; \node[right] at (1.5,0) {$p$};
        \draw[red, thick, mark=square*, mark options={red}] plot coordinates {(3,0.1) (3.5,0.1)}; \node[right] at (3.5,0) {$T$};
        \draw[green, thick, mark=triangle*, mark options={green}] plot coordinates {(5,0.1) (5.5,0.1)}; \node[right] at (5.5,0) {$\bu$};
        
        \draw[orange, thick, mark=*, mark options={orange}] plot coordinates {(7,0.1) (7.5,0.1)}; \node[right] at (7.5,0) {$\nabla p$};
        \draw[purple, thick, mark=square*, mark options={purple}] plot coordinates {(9,0.1) (9.5,0.1)}; \node[right] at (9.5,0) {$\nabla T$};
        \draw[brown, thick, mark=triangle*, mark options={brown}] plot coordinates {(11,0.1) (11.5,0.1)}; \node[right] at (11.5,0) {$\nabla \bu$};
    \end{tikzpicture}
    
    \begin{minipage}{0.42\textwidth}
        \centering
        \begin{tikzpicture}[scale=0.8]
            \begin{loglogaxis}[
                title={S-based Upwind scheme - $\mathbb{K} = \mathbb{I}$ },
                ylabel= $L^2$ Relative Errors,
                xlabel= $h$,
            ]
            \addplot[blue, mark=*] table {data/ErrorP-entropy-up.dat};
            \addplot[red, mark=square*] table {data/ErrorT-entropy-up.dat};
            \addplot[green, mark=triangle*] table {data/ErrorU-entropy-up.dat};
            \addplot[orange, mark=*] table {data/ErrordP-entropy-up.dat};
            \addplot[purple, mark=square*] table {data/ErrordT-entropy-up.dat};
            \addplot[brown, mark=triangle*] table {data/ErrordU-entropy-up.dat};
            \logLogSlopeTriangle{0.75}{0.4}{0.07}{1}{black};
            \logLogSlopeTriangle{0.75}{0.4}{0.07}{2}{black};

            \end{loglogaxis}
        \end{tikzpicture}
    \end{minipage}
    \hfill
    \begin{minipage}{0.42\textwidth}
        \centering
        \begin{tikzpicture}[scale=0.8]
            \begin{loglogaxis}[
                title={S-based Centred scheme},
                ylabel= ,
                xlabel= $h$,
            ]
            \addplot[blue, mark=*] table {data/ErrorP-entropy-c.dat};
            \addplot[red, mark=square*] table {data/ErrorT-entropy-c.dat};
            \addplot[green, mark=triangle*] table {data/ErrorU-entropy-c.dat};
            \addplot[orange, mark=*] table {data/ErrordP-entropy-c.dat};
            \addplot[purple, mark=square*] table {data/ErrordT-entropy-c.dat};
            \addplot[brown, mark=triangle*] table {data/ErrordU-entropy-c.dat};
            \logLogSlopeTriangle{0.75}{0.4}{0.07}{1}{black};
            \logLogSlopeTriangle{0.75}{0.4}{0.07}{2}{black};
            \end{loglogaxis}
        \end{tikzpicture}
    \end{minipage}

    \vspace{1em}  

    \begin{minipage}{0.42\textwidth}
        \centering
        \begin{tikzpicture}[scale=0.8]
            \begin{loglogaxis}[
                title={S-based Upwind scheme - $\mathbb{K} = 100 ~\mathbb{I}$  },
                ylabel= $L^2$ Relative Errors,
                xlabel= $h$,
            ]
            \addplot[blue, mark=*] table {data/ErrorP-entropy-up-100.dat};
            \addplot[red, mark=square*] table {data/ErrorT-entropy-up-100.dat};
            \addplot[green, mark=triangle*] table {data/ErrorU-entropy-up-100.dat};
            \addplot[orange, mark=*] table {data/ErrordP-entropy-up-100.dat};
            \addplot[purple, mark=square*] table {data/ErrordT-entropy-up-100.dat};
            \addplot[brown, mark=triangle*] table {data/ErrordU-entropy-up-100.dat};
            \logLogSlopeTriangle{0.75}{0.4}{0.07}{1}{black};
            \logLogSlopeTriangle{0.75}{0.4}{0.07}{2}{black};

            \end{loglogaxis}
        \end{tikzpicture}
    \end{minipage}

   \caption{Convergence of the relative $L^2$ errors for the temperature $T$, pressure $p$, and displacement $\bu$ and their gradients for the  discretisation of the \eqref{model:entropy}  model and both the centred (for $\mathbb{K} = \mathbb{I}$) and upwind (for $\mathbb{K} = \mathbb{I},100 ~\mathbb{I}$) schemes.}

  \label{fig:convergenceSansFracS-based}
\end{figure}

\begin{figure}[H]
    \centering
    
    \begin{tikzpicture}[scale=0.8]
        \node at (0,0) { };
        
        \draw[blue, thick, mark=*, mark options={blue}] plot coordinates {(1,0.1) (1.5,0.1)}; \node[right] at (1.5,0) {$p$};
        \draw[red, thick, mark=square*, mark options={red}] plot coordinates {(3,0.1) (3.5,0.1)}; \node[right] at (3.5,0) {$T$};
        \draw[green, thick, mark=triangle*, mark options={green}] plot coordinates {(5,0.1) (5.5,0.1)}; \node[right] at (5.5,0) {$\bu$};
        
        \draw[orange, thick, mark=*, mark options={orange}] plot coordinates {(7,0.1) (7.5,0.1)}; \node[right] at (7.5,0) {$\nabla p$};
        \draw[purple, thick, mark=square*, mark options={purple}] plot coordinates {(9,0.1) (9.5,0.1)}; \node[right] at (9.5,0) {$\nabla T$};
        \draw[brown, thick, mark=triangle*, mark options={brown}] plot coordinates {(11,0.1) (11.5,0.1)}; \node[right] at (11.5,0) {$\nabla \bu$};
    \end{tikzpicture}
    
    \begin{minipage}{0.48\textwidth}
        \centering
        \begin{tikzpicture}[scale=0.8]
            \begin{loglogaxis}[
                title={H-based Upwind scheme - $\mathbb{K} = \mathbb{I}$},
                ylabel= $L^2$ Relative Errors,
                xlabel= $h$,
            ]
            \addplot[blue, mark=*] table {data/ErrorP-enthalpy-up.dat};
            \addplot[red, mark=square*] table {data/ErrorT-enthalpy-up.dat};
            \addplot[green, mark=triangle*] table {data/ErrorU-enthalpy-up.dat};
            \addplot[orange, mark=*] table {data/ErrordP-enthalpy-up.dat};
            \addplot[purple, mark=square*] table {data/ErrordT-enthalpy-up.dat};
            \addplot[brown, mark=triangle*] table {data/ErrordU-enthalpy-up.dat};
            \logLogSlopeTriangle{0.75}{0.4}{0.07}{1}{black};
            \logLogSlopeTriangle{0.75}{0.4}{0.07}{2}{black};

            \end{loglogaxis}
        \end{tikzpicture}
    \end{minipage}
    \hfill
    \begin{minipage}{0.48\textwidth}
        \centering
        \begin{tikzpicture}[scale=0.8]
            \begin{loglogaxis}[
                title={H-based Centred scheme },
                ylabel= $L^2$ Relative Error,
                xlabel= $h$,
            ]
            \addplot[blue, mark=*] table {data/ErrorP-enthalpy-c.dat};
            \addplot[red, mark=square*] table {data/ErrorT-enthalpy-c.dat};
            \addplot[green, mark=triangle*] table {data/ErrorU-enthalpy-c.dat};
            \addplot[orange, mark=*] table {data/ErrordP-enthalpy-c.dat};
            \addplot[purple, mark=square*] table {data/ErrordT-enthalpy-c.dat};
            \addplot[brown, mark=triangle*] table {data/ErrordU-enthalpy-c.dat};
            \logLogSlopeTriangle{0.75}{0.4}{0.07}{1}{black};
            \logLogSlopeTriangle{0.75}{0.4}{0.07}{2}{black};

            \end{loglogaxis}
        \end{tikzpicture}
    \end{minipage}

    \vspace{1em}  

    \begin{minipage}{0.48\textwidth}
        \centering
        \begin{tikzpicture}[scale=0.8]
            \begin{loglogaxis}[
                title={H-based Upwind scheme - $\mathbb{K} = 100 ~\mathbb{I}$  },
                ylabel= $L^2$ Relative Errors,
                xlabel= $h$,
            ]
            \addplot[blue, mark=*] table {data/ErrorP-enthalpy-up-100.dat};
            \addplot[red, mark=square*] table {data/ErrorT-enthalpy-up-100.dat};
            \addplot[green, mark=triangle*] table {data/ErrorU-enthalpy-up-100.dat};
            \addplot[orange, mark=*] table {data/ErrordP-enthalpy-up-100.dat};
            \addplot[purple, mark=square*] table {data/ErrordT-enthalpy-up-100.dat};
            \addplot[brown, mark=triangle*] table {data/ErrordU-enthalpy-up-100.dat};
            \logLogSlopeTriangle{0.75}{0.4}{0.07}{1}{black};
            \logLogSlopeTriangle{0.75}{0.4}{0.07}{2}{black};

            \end{loglogaxis} 
        \end{tikzpicture}
    \end{minipage}

    \caption{Convergence of the relative $L^2$ errors for the temperature $T$, pressure $p$, and displacement $\bu$ and their gradients for the  discretisation of the \eqref{model:enthalpy} model and both the centred (for $\mathbb{K} = \mathbb{I}$) and upwind (for $\mathbb{K} = \mathbb{I},100 ~\mathbb{I}$) schemes.}

   \label{fig:convergenceSansFracH-based}
\end{figure}
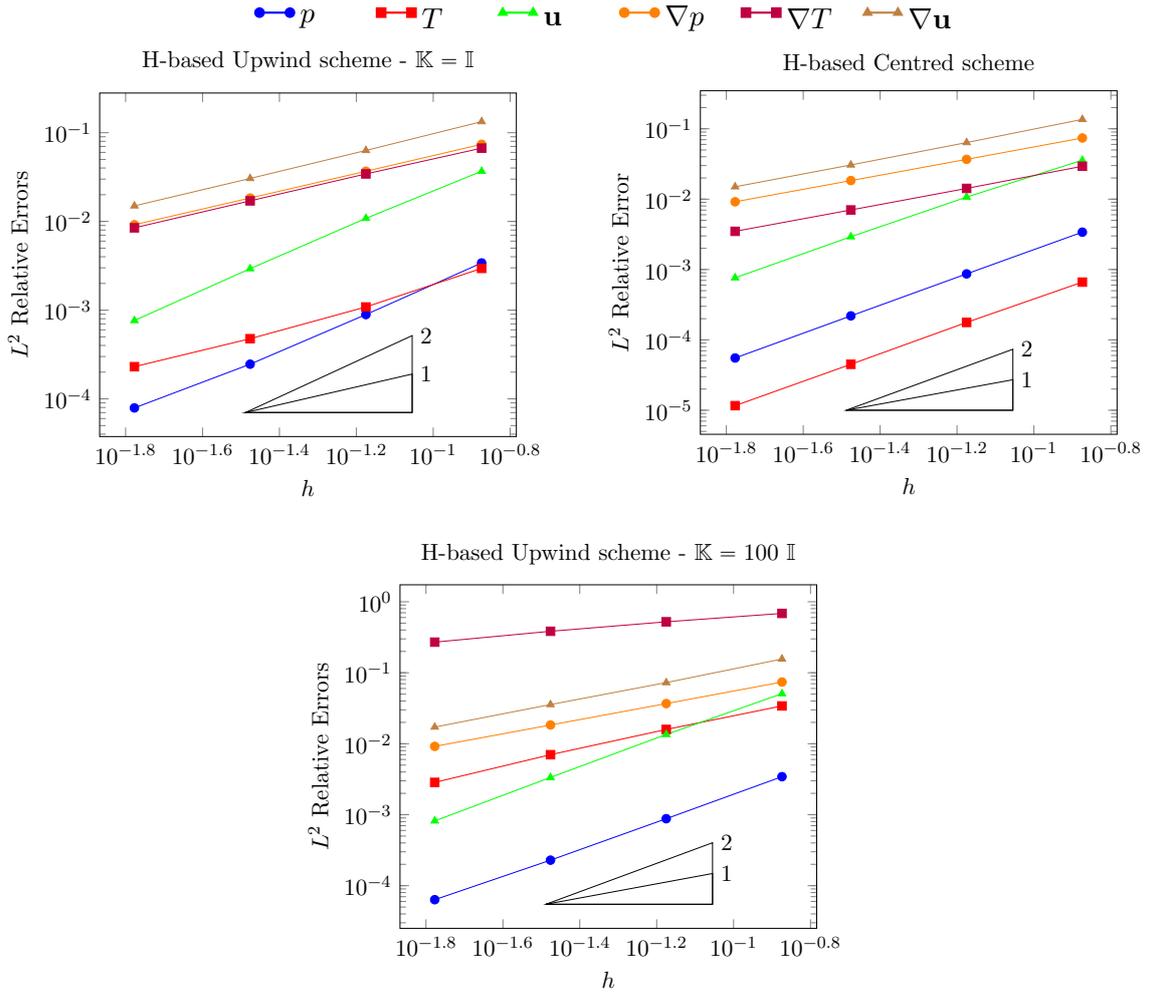

\subsection{2D Discrete Fracture Matrix (DFM) model}\label{sec:testdfm}

We consider the DFM model exhibited in Figure \ref{fig:fractures-net}  and introduced in \cite{contact-norvegiens}. It is defined on the domain $\Omega = (0,1)\times (0,2)\,\mathrm{m}^2$ and includes a network $\Gamma$ of six fractures, among which fracture $1$ made up of two sub-fractures forming a corner,  nearly intersecting fractures and the non-immersed fracture $5$ with one tip lying on the left boundary. 
The setting of the simulation follows the test case proposed in \cite{STEFANSSON2021114122} and consists in three stages on the time intervals $I_1 = (0, t^{(1)})$, $I_2 = [t^{(1)}, t^{(2)})$ and $I_3 = [t^{(2)}, t_F]$. These three stages are monitored by the boundary conditions triggering respectively mechanical, hydraulic and thermal driving forces. As exhibited in Figure \ref{fig:3stages}, the displacement $\bu=\prescript{t}{}(5,-2) \,10^{-4}\,\mathrm{m}$ is imposed at the top boundary for $t>0$. At the left boundary, a high pressure $p=8.~10^{6} \, \mathrm{Pa}$ is set for $t\geq t^{(1)}$, and a low temperature $T=285\, \mathrm{K}$ is prescribed for $t\geq t^{(2)}$.

The initial pressure and temperature are fixed to 
$p^0 = 10^5 \, \mathrm{Pa}$ and $T^0 = 300\, \mathrm{K}$. The initial displacement field $\bu^0$ is computed by solving the discrete contact-mechanical problem \eqref{Mixed_meca_eq} at given $p^0,T^0$ with a zero displacement $\bu=\mathbf{0}$ at the bottom and top boundaries and a free boundary condition at the left and right sides.
Throughout the simulation, the top and bottom boundaries are assumed impervious with zero heat flux while the right boundary condition imposes a fixed pressure $p=10^5 \, \mathrm{Pa}$ and temperature $T=300\, \mathrm{K}$. A zero displacement $\bu = \mathbf{0}$ is prescribed at the bottom boundary and the left and right sides are kept with homogeneous total stress conditions.

In order to compare the \eqref{model:enthalpy} and \eqref{model:entropy} models and their discretisations on different fluid behaviors, we consider in the following two cases corresponding first to a slightly compressible liquid and second to a perfect gas. The simulation parameters, not depending on the fluid thermodynamical model, are reported in Table \ref{TB2} together with $\alpha_\phi = (b - \phi^0) \alpha_s$, $K_s = \lambda + \mu$, $N= {K_s \over (b-\phi^0) (1-b)}$. The times $t^{(1)}$,  $t^{(2)}$,  $t_F$ will be fixed according to the fluid thermodynamical model in such a way that a stationary state is roughly reached at the end of each stage.

The simulations are performed on a family of 4 uniformly refined meshes indexed by $m\in \{0,\cdots,3\}$ with 
$2855\times 4^m$ triangular cells and $88\times 2^m$ fracture faces. The finest mesh $m=3$ is used for the reference solution in the numerical convergence investigations. Due to the thermal convection dominated regime during stage 3, the convection terms are upwinded for the discretisation of both models.

\begin{figure}[!htb]
    \centering
    \includegraphics[width= 7cm]{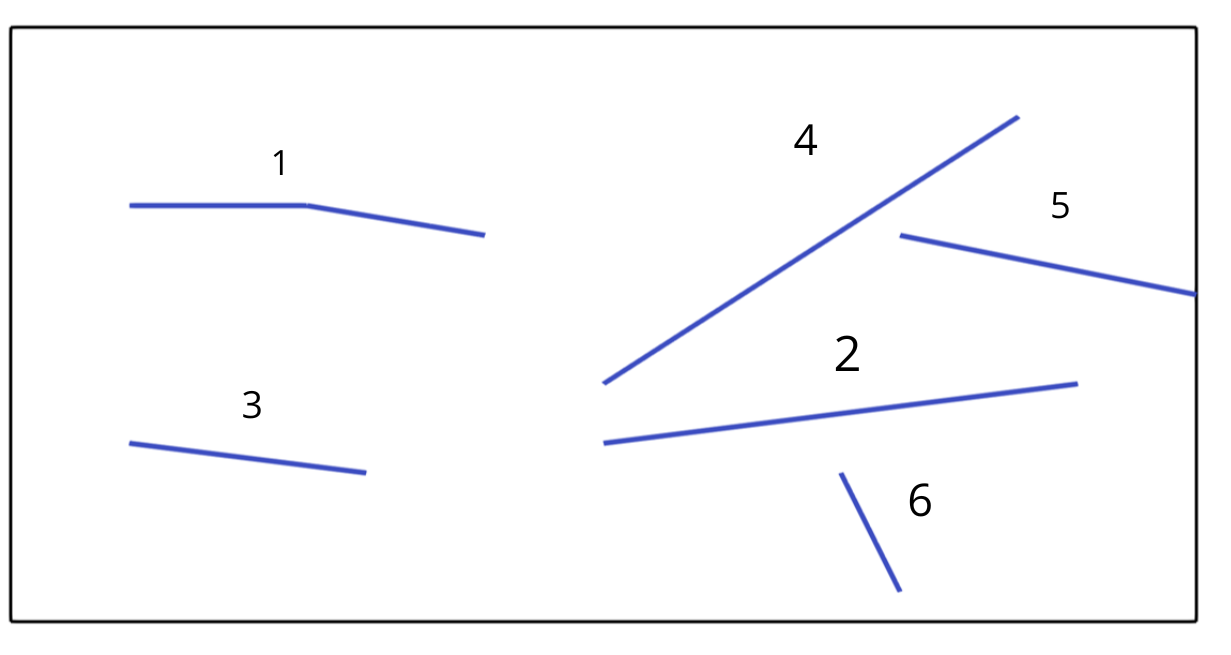}
    \caption{Two-dimensional domain $\Omega = (0,1) \times ( 0,2) \,\mathrm{m}^2$ including the six fracture network $\Gamma$.}
    \label{fig:fractures-net}
\end{figure}

\begin{table}[H]
 \caption{Material Properties common to the liquid and gas test cases.}
\centering
{
 \resizebox{12cm}{!}{
$\begin{array}{llll}
\hline \text { Symbol } & \text { Quantity } & \text { Value } & \text { Unit } \\
\hline E & \text { Young modulus } & 40   & \mathrm{~GPa} \\
\nu & \text { Poisson coefficient } & 0.15 & - \\
F & \text { Friction coefficient } & 0.5 & - \\
b & \text { Biot coefficient } & 0.65 & - \\
 \mathbb{K} & \text{Permeability coefficient} & \begin{pmatrix} 1 & 0 \\ 0 & 0.5 \end{pmatrix} \times 10^{-15} & \mathrm{m}^2 \\
 \phi^0 & \text { Initial porosity} &  0.1  & -   \\
 d_0 & \text { Contact aperture} &  5.~10^{-4}  & \mathrm{~m}   \\
\Lambda_m & \text { Effective thermal conductivity } & 2 & \mathrm{W} \mathrm{~m}^{-1} \mathrm{~K}^{-1} \\
\alpha_s & \text { The volumetric skeleton thermal dilation coefficient } & 1.5~10^{-5} &  \mathrm{~K}^{-1} \\
m_0 & \text { Average fluid skeleton specific density} & 0 & \mathrm{~Kg~m^{-3}} \\
C_s & \text { The skeleton volumetric heat capacity } & 2 & \mathrm{~MJ} \mathrm{~m}^{-3} \mathrm{~K}^{-1} \\
\hline
\end{array}$
}
}
\label{TB2}
\end{table}

\begin{figure}[H]
    \centering
    
    \begin{minipage}[b]{\linewidth}
     \centering
      (a) Stage 1 \hskip 5cm  (b) Stage 2 \\
     \includegraphics[width=0.45\linewidth]{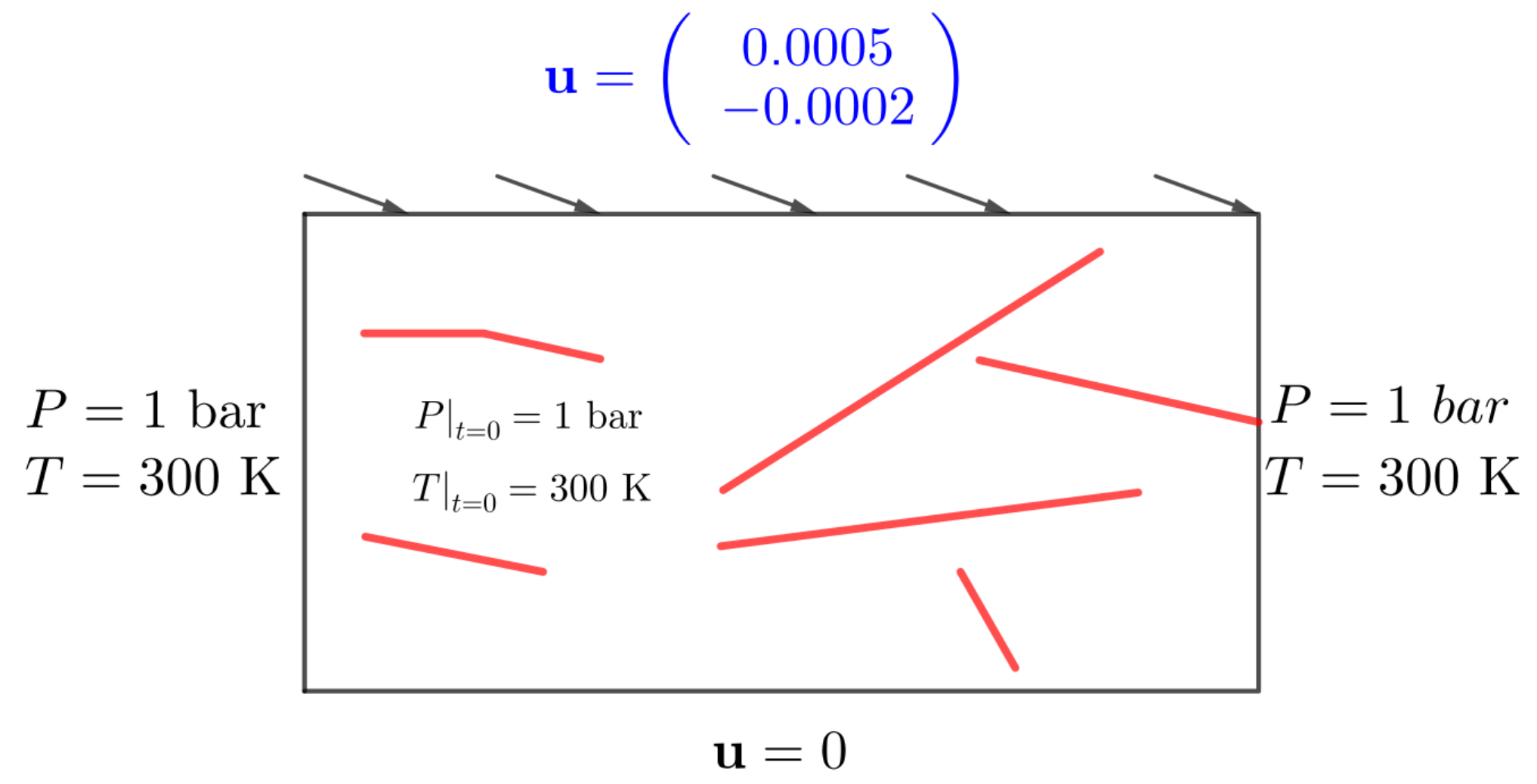}\hskip 0.5cm 
     \includegraphics[width=0.45\linewidth]{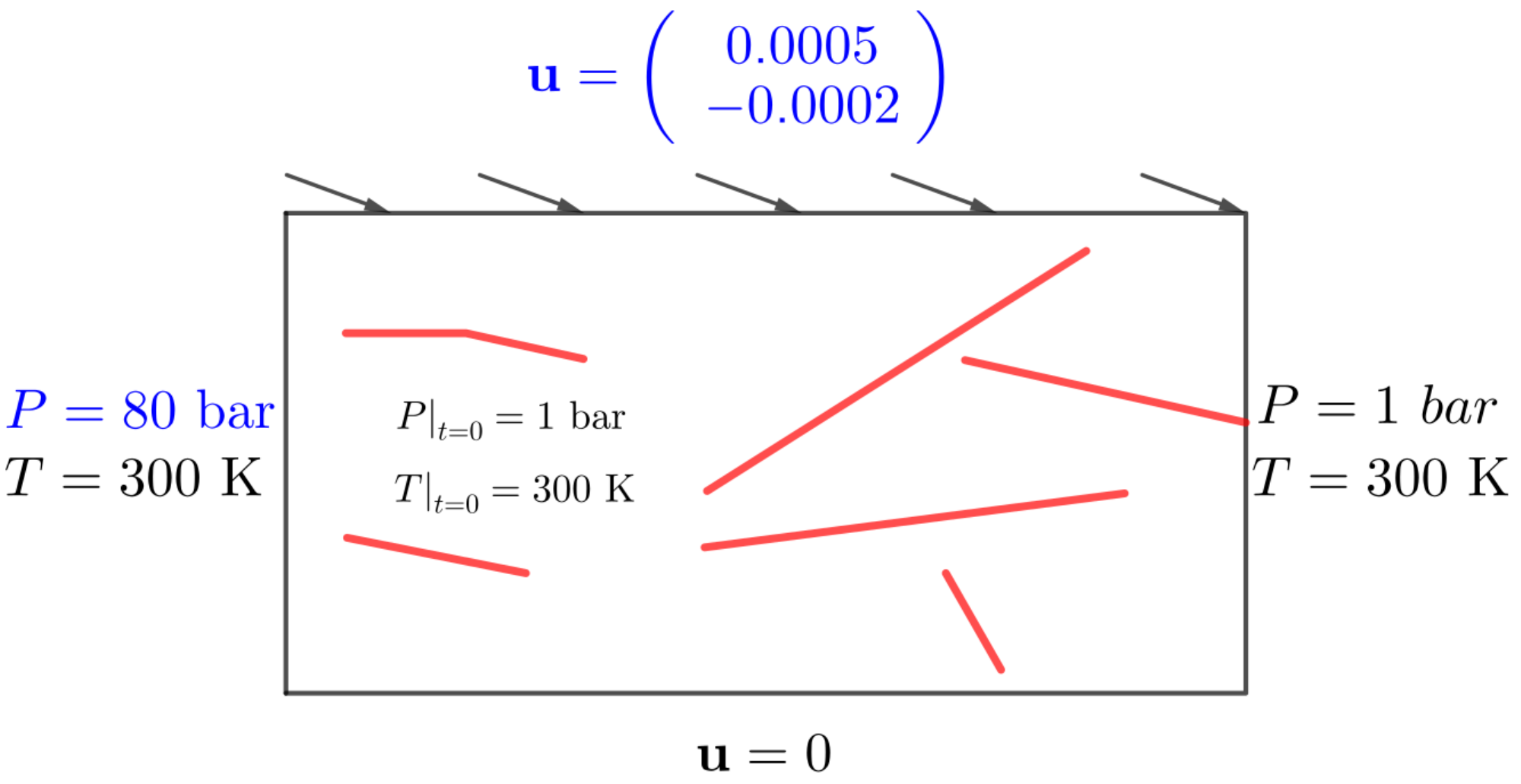}
  \end{minipage}
    \begin{minipage}[b]{\linewidth}
      \centering
      \vskip -0.25cm 
        (c) Stage 3 \\
        \includegraphics[width=0.45\linewidth]{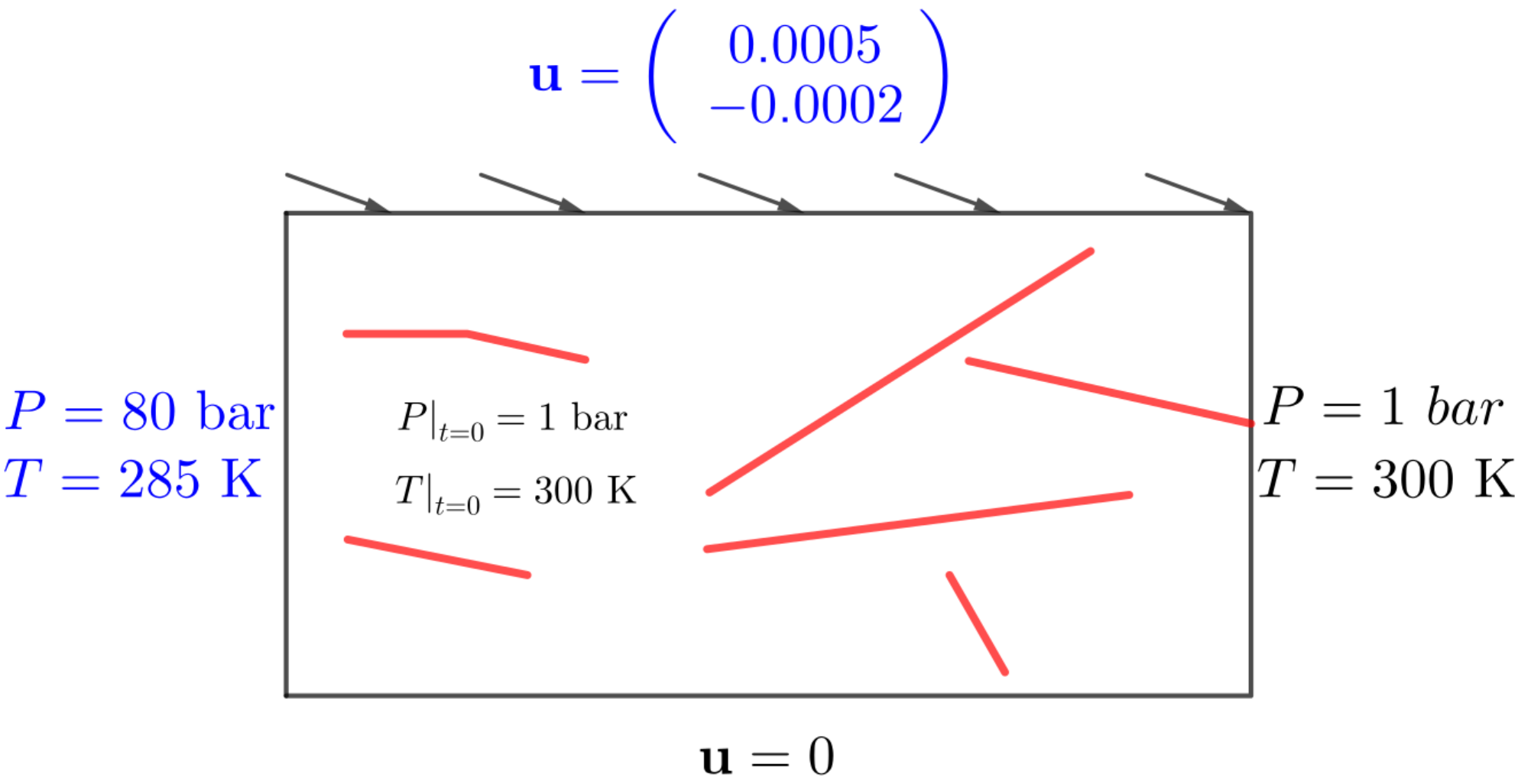}
    \end{minipage}

    \caption{Set up of the simulation in terms of initial condition at time $t=0$ and of boundary conditions during each of the three stages corresponding to the time intervals $I_1 = (0, t^{(1)})$, $I_2 = [t^{(1)}, t^{(2)})$ and $I_3 = [t^{(2)}, t_F]$. Here $1$ bar corresponds to $10^5$ Pa.}
    \label{fig:3stages}
\end{figure}

\subsubsection{Weakly compressible liquid case} \label{sec:testliquid}

We consider a weakly compressible liquid with thermodynamical constitutive laws deriving from a free enthalpy potential. It is characterised by its specific density $\varrho(p,T)$ such that 
$$
{\rhoref \over \varrho(p,T) } = 1 - \frac{(p-\pref)}{K_f}+ \alpha_f (T-\Tref),
$$
and its specific internal energy  
$$
e(p,T) = C_f T  - {\alpha_f \over \rhoref} \( (p-\pref)\Tref + p (T-\Tref) \) + {(p^2 - \pref^2)\over 2 \rhoref K_f},
$$
given the parameters $\Tref = 300$ K, $\pref = 10^5$ Pa, $\rhoref = 10^3$ Kg.m$^{-3}$, $K_f = 2.18$ GPa, $\alpha_f = 2.07 ~  10^{-4} K^{-1}$, and $C_f = 4180$ J.Kg$^{-1}$.K$^{-1}$. The fluid viscosity is set to $\eta= 10^{-3} \mathrm{~Pa} \mathrm{~s}$.  
The time intervals for each of the three stages are given by $t^{(1)} = 100$ s, $t^{(2)} = 200$ s and $t_F = 5$ days. The time stepping is defined  by a small initial time step of $0.1$ s in order to capture the undrained regime time scale at the beginning of stage 1 and by the maximum time steps of $5$ s for stage 1 and 2 and of $0.1$ day for stage 3.

Figure \ref{fig:contactstateLiquid} exhibits the evolution of the contact state (open, contact stick or contact slip) along the fractures at different times during the three stages. At time $t>0$, due to the imposed displacement at the top, most of the fractures switch from open to contact. During the undrained regime, at the very beginning of stage 1, the high increase of the pressure (see Figure \ref{fig:3stagessolutions} (a)) induces a slip state for most of the fractures in contact, as a result of the reduction of the normal surface traction. Toward the end of stage 1, these fractures switch back to stick state due to the pressure relaxation (see the evolution of the mean pressure in Figure \ref{fig:meanVARIATIONSLiquid} (a)). During stage 2, fractures 1 and 3 switch back to slip state as a result of the high pressure front propagation (see Figure \ref{fig:3stagessolutions} (b) and \ref{fig:meanVARIATIONSLiquid} (a)) while they partially or totally open during the cold temperature front propagation in stage 3 (see Figures \ref{fig:3stagessolutions} (c) and \ref{fig:meanVARIATIONSLiquid} (b)) as a result of the matrix shrinkage.

\begin{figure}[H]
    \centering
    
    \begin{minipage}[b]{\linewidth}
     \centering
      (a) Stage 1 \hskip 5cm (b) Stage 2 \\
     \includegraphics[width=0.45\linewidth]{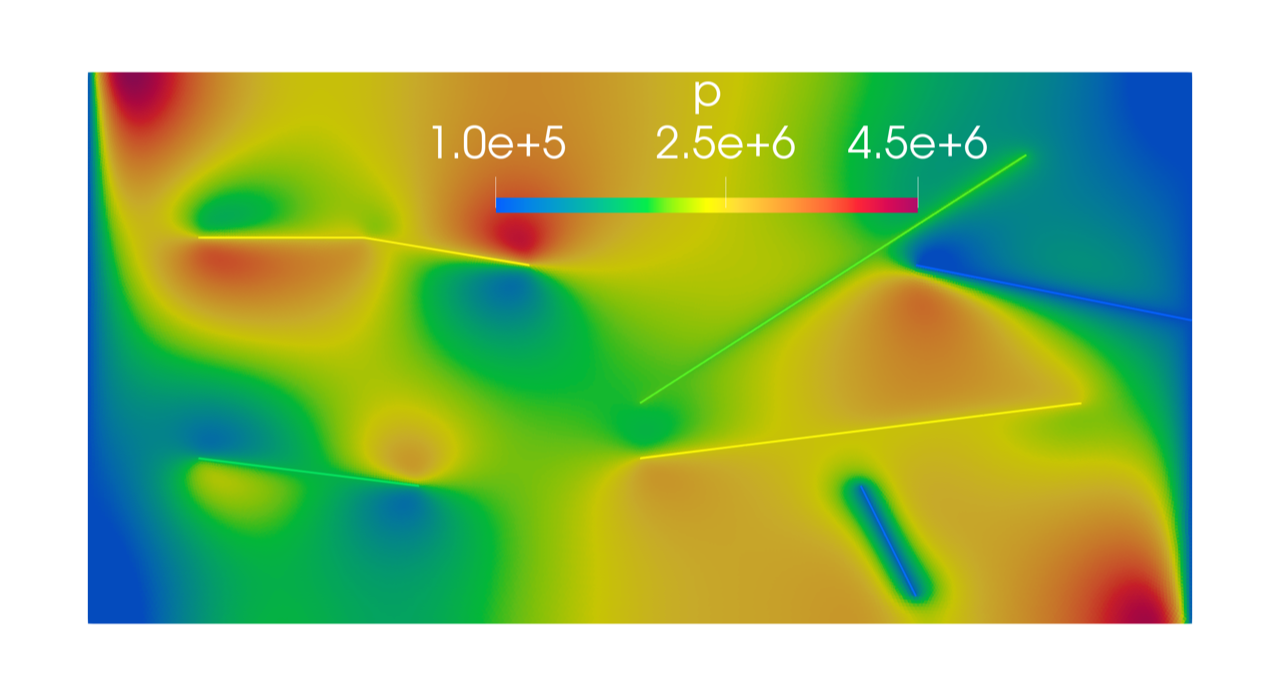}
    \includegraphics[width=0.45\linewidth]{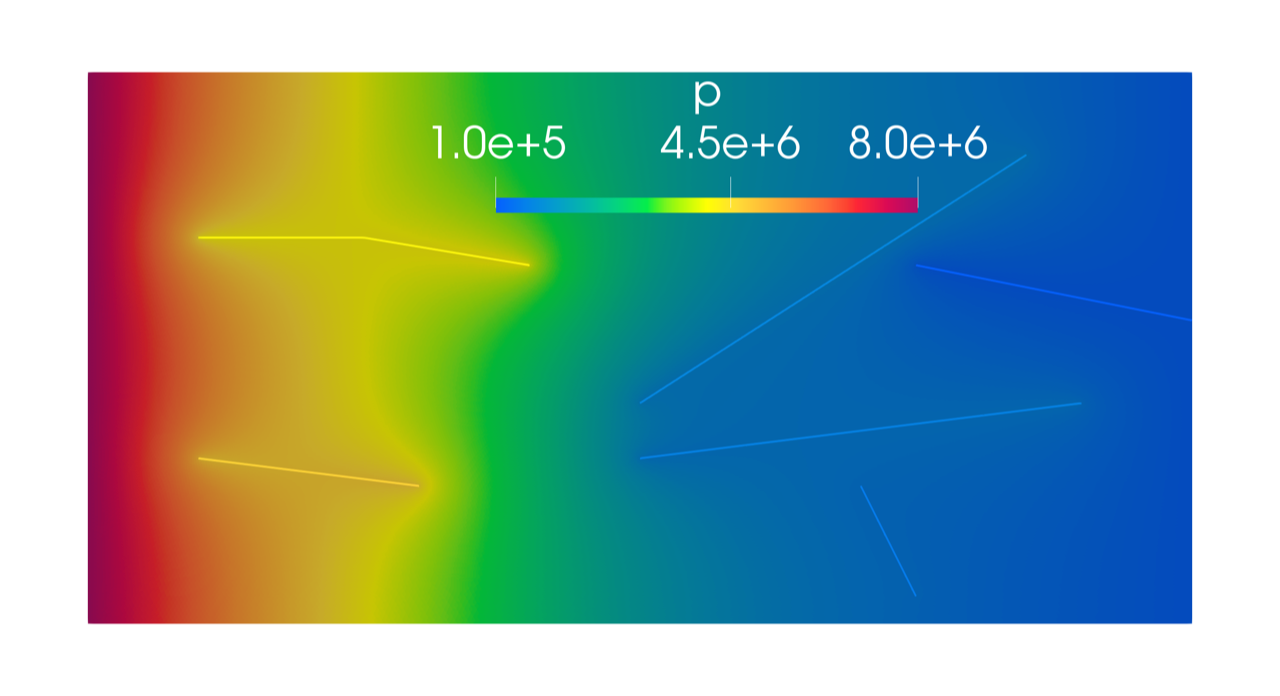}
  \end{minipage}
    \begin{minipage}[b]{\linewidth}
    \centering
        (c) Stage 3 \\
        \includegraphics[width=0.45\linewidth]{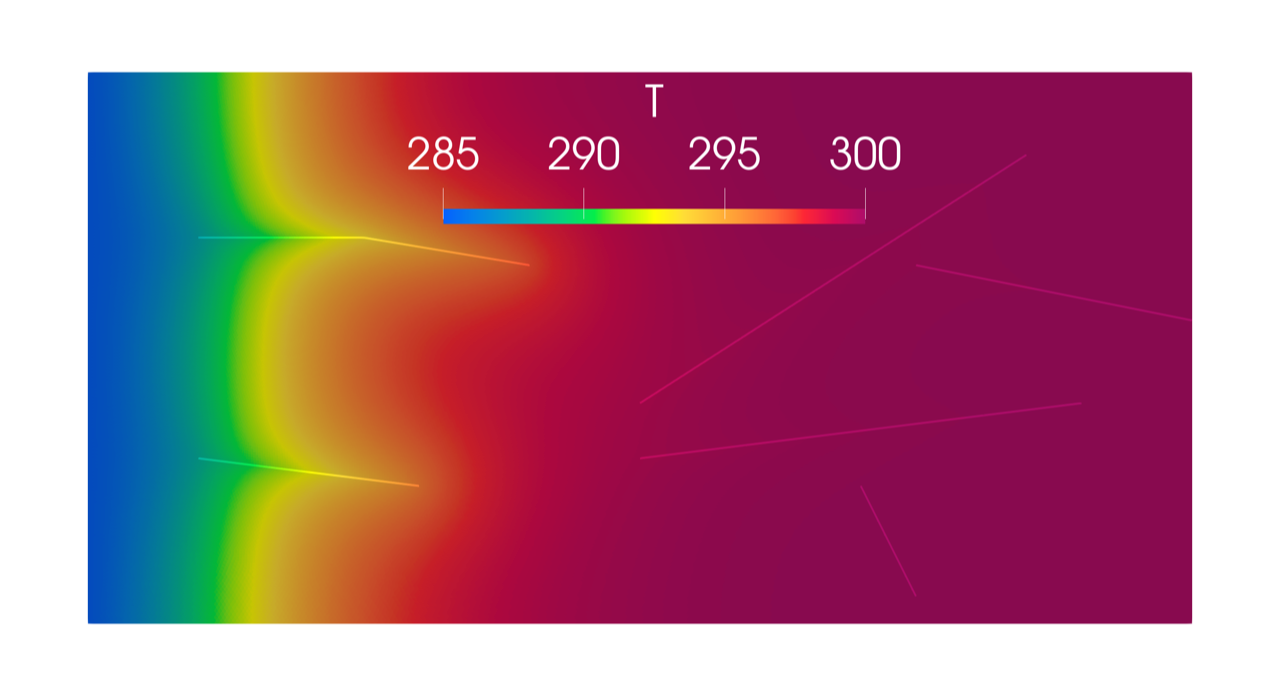}
    \end{minipage}

    \caption{Discrete solution obtained with the enthalpy-based model on the finest mesh $m=3$: (a) pressure $p$ at the very beginning of stage 1 ($t=0.1$ s), (b) pressure $p$ at $t=116$ s during stage 2, (c) temperature $T$ at $t=27715$ s  time during stage 3.}
    \label{fig:3stagessolutions}
\end{figure}


\begin{figure}[!htb]
    \centering
        \begin{minipage}[b]{0.45\linewidth}
        \centering
        $\text {(a) Initial state }(p=1 ~\mathrm{bar}, T=300 \mathrm{~K})$ \\[0.25cm]
        \includegraphics[width=.8\linewidth]{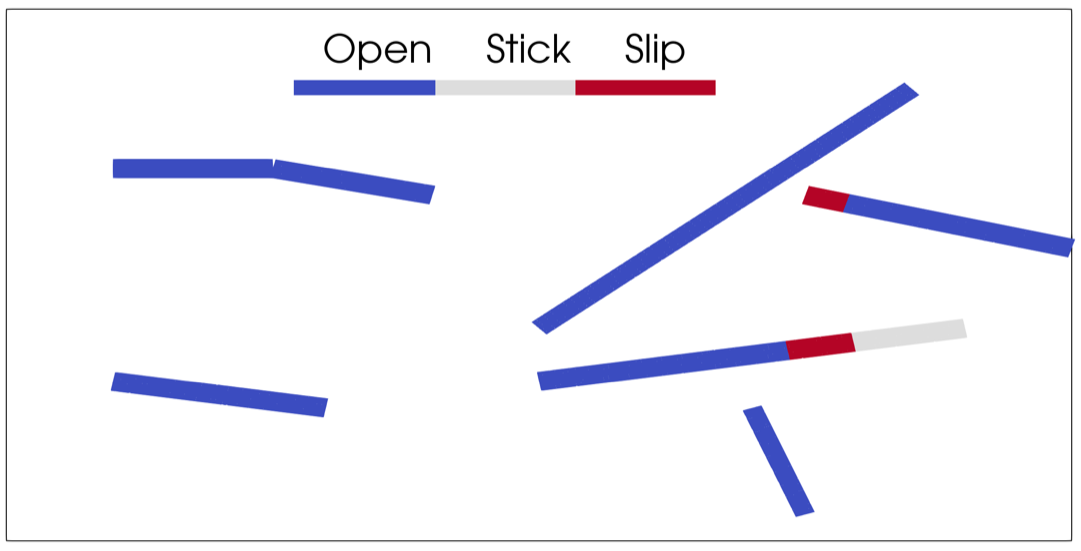}
         \end{minipage}

     \vspace{0.5cm}
        \begin{minipage}[b]{0.45\linewidth}
     \centering
        (b) Beginning of Stage 1 \\[0.25cm]
        \includegraphics[width=.8\linewidth]{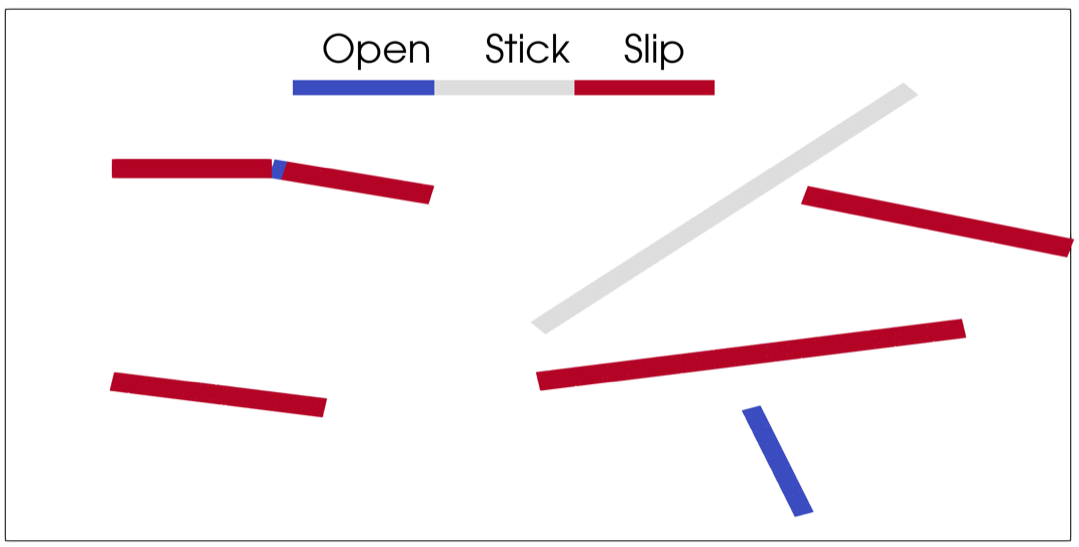}
  \end{minipage}
    \hfill
    \begin{minipage}[b]{0.45\linewidth}
      \centering
       (c) End of stage 1 \\[0.25cm]
        \includegraphics[width=.8\linewidth]{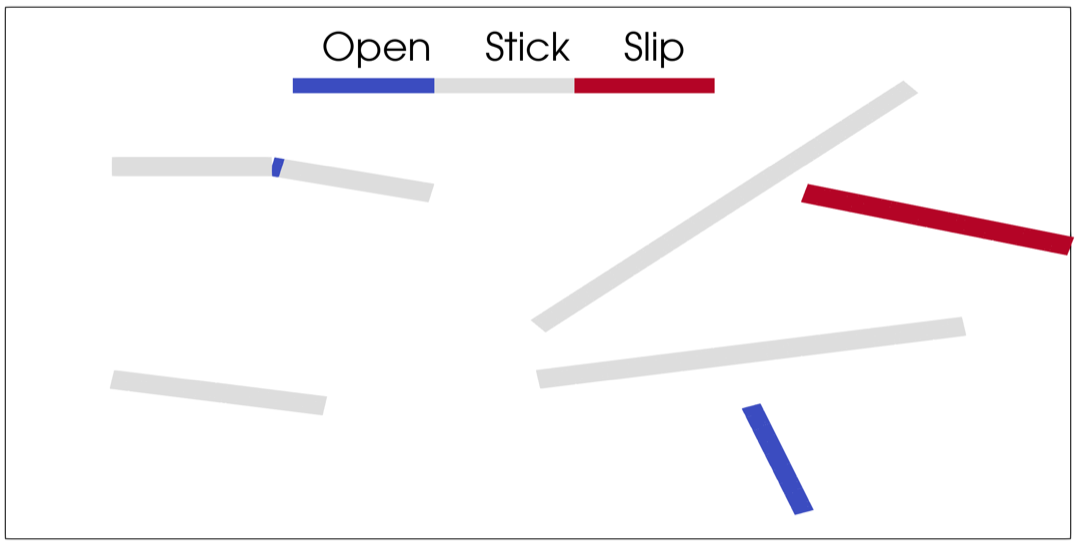}
    \end{minipage}

    \vspace{0.5cm}

    \begin{minipage}[b]{0.45\linewidth}
    \centering
        (d) Stage 2 \\[0.25cm]
        \includegraphics[width=.8\linewidth]{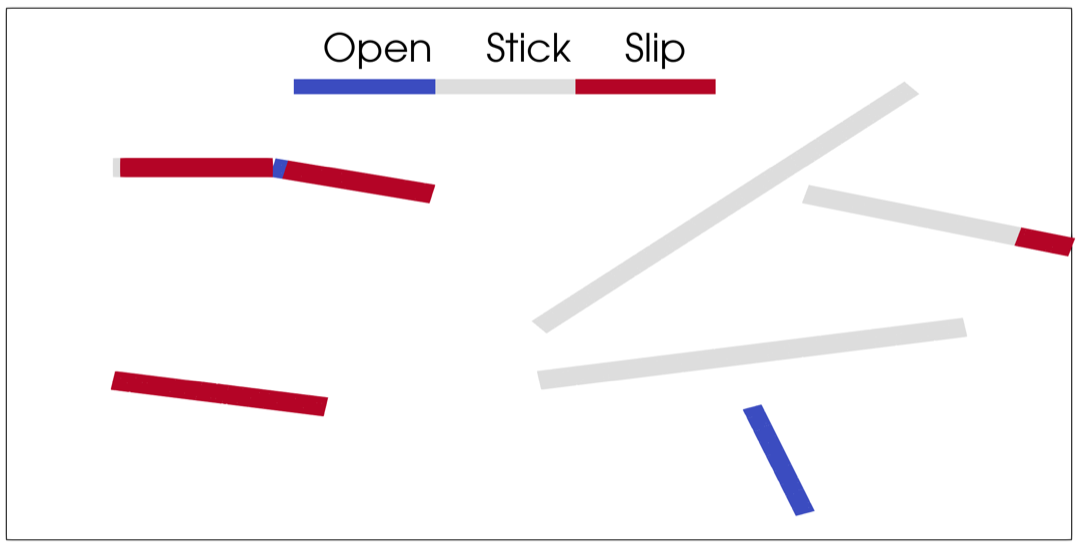}
    \end{minipage}
    \hfill
      \begin{minipage}[b]{0.45\linewidth}
      \centering
       (e) Stage 3 \\[0.25cm]
        \includegraphics[width=.8\linewidth]{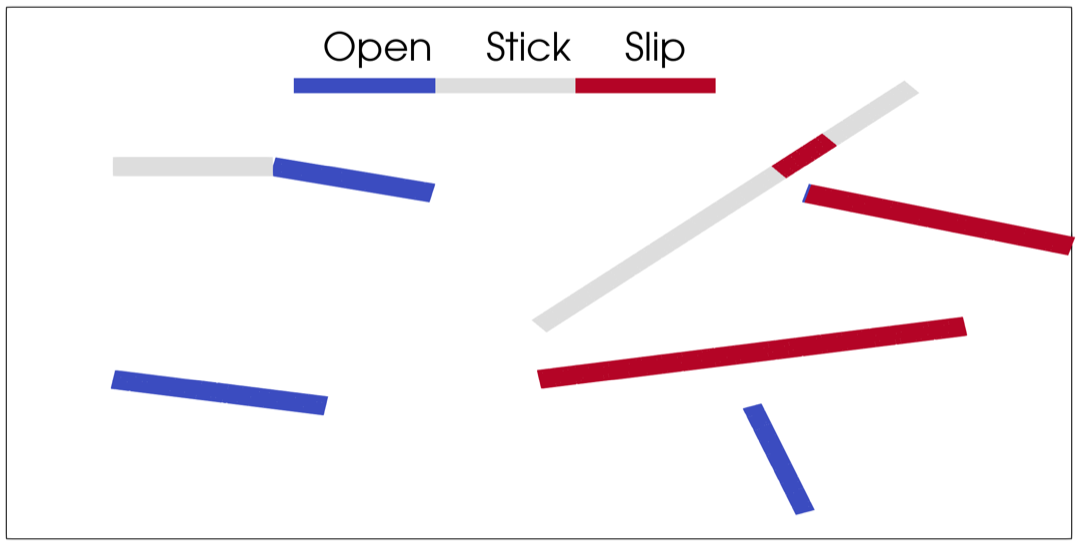}
    \end{minipage}

    \caption{Contact state (open, contact stick, contact slip) along the fractures at different times for the discrete solution on the finest mesh $m=3$ of the enthalpy-based model.}

    \label{fig:contactstateLiquid}
\end{figure}


Figure \ref{fig:meanVARIATIONSLiquid} compares the evolution in time of mean in space variables ($p$, $p_f$, $T$, $T_f$, and scaled $d_f-d_0$, $\jump{\bu}_\tau$, and $\phi-\phi^0$) for the entropy-based and enthalpy-based discrete solutions obtained on the finest mesh $m=3$. A very good match is observed on all variables with only small differences during stage 3 which has been checked to result from the neglected terms $\bV\cdot \nabla p$ in the approximate entropy equations of the \eqref{model:entropy} model. When these terms are added back, the discrete solutions of both models fully match. Note that the linearisation of the Fourier term is not a significant source of discrepancy for this test case due to the thermal convection dominated regime.

Figure \ref{fig:convergenceratesLiquid} compares, for the discrete solutions of both models, the convergence of the $L^2$ errors in time of mean in space variables as functions of the mesh step $h$ (1/total number of fracture faces). The errors are computed w.r.t.~the fine mesh $m=3$ reference solutions obtained using the same time stepping and the same model.  The observed convergence behaviors of the discretisations of the two models are very similar, with a rate of convergence roughly  equal to $1.5$.


\begin{figure}[h]
    \centering
    
    \begin{minipage}{0.45\textwidth}
        \centering
        \begin{tikzpicture}[scale=0.8]
                   \begin{axis}[
            title={},
            xlabel={Time (s)},
            ylabel={mean Pressure},
            grid=major,
            xmode=log, 
           legend pos=south east
        ]
        
        \addplot+[no marks] table[x index=0, y index=2, col sep=space] {data/Tm_Pm_Phim.txt};
        \addlegendentry{H-$p$};
        
        \addplot+[no marks] table[x index=0, y index=2, col sep=space] {data/df_pf_sf_slip.txt};
        \addlegendentry{H-$p_f$};
        
        \addplot+[dashed, no marks, color=black] table[x index=0, y index=2, col sep=space] {data/Tm_Pm_Phim_S.txt};
        \addlegendentry{S-$p$};
        
        \addplot+[dashed, no marks] table[x index=0, y index=2, col sep=space] {data/df_pf_sf_slip_S.txt};
        \addlegendentry{S-$p_f$};
        
        \end{axis}
        \end{tikzpicture}
        (a) 
    \end{minipage}
    \hfill
    \begin{minipage}{0.45\textwidth}
      \centering
      \vskip 0.5cm 
        \begin{tikzpicture}[scale=0.8]
               \begin{axis}[
            title={},
            xlabel={Time (s)},
            ylabel={mean Temperatures},
            grid=major,
            xmode=log, 
            legend pos=south west 
        ]
        
        \addplot+[no marks] table[x index=0, y index=1, col sep=space] {data/Tm_Pm_Phim.txt};
        \addlegendentry{H-$T$};
        
        \addplot+[no marks] table[x index=0, y index=3, col sep=space] {data/df_pf_sf_slip.txt};
        \addlegendentry{H-$T_f$};
        
        \addplot+[dashed, no marks, color=green] table[x index=0, y index=1, col sep=space] {data/Tm_Pm_Phim_S.txt};
        \addlegendentry{S-$T$};
        
        \addplot+[dashed, no marks] table[x index=0, y index=3, col sep=space] {data/df_pf_sf_slip_S.txt};
        \addlegendentry{S-$T_f$};
        
        \end{axis}
        \end{tikzpicture}
        (b) 
    \end{minipage}

    \vspace{1em}  

    \begin{minipage}{0.45\textwidth}
        \centering
        \begin{tikzpicture}[scale=0.8]
                \begin{axis}[
            title={},
            xlabel={Time (s)},
            ylabel={Scaled mean variations},
            grid=major,
            xmode=log, 
            legend pos=outer north east,
        ]
        
        \addplot+[no marks] table[x index=0, y index=1, col sep=space] {data/dfmd0_mesh3.txt};
        \addlegendentry{H-$d_f$};
        
        \addplot+[no marks] table[x index=0, y index=1, col sep=space] {data/slipf_mesh3.txt};
        \addlegendentry{H-$\llbracket \bu \rrbracket_\tau$};

        \addplot+[no marks] table[x index=0, y index=1, col sep=space] {data/dphim_mesh3.txt};
        \addlegendentry{H-$\phi$};
        
        \addplot+[dashed, no marks, color=green] table[x index=0, y index=1, col sep=space] {data/dfmd0_mesh3_S.txt};
        \addlegendentry{S-$d_f$};
        
        \addplot+[dashed, no marks] table[x index=0, y index=1, col sep=space] {data/slipf_mesh3_S.txt};
        \addlegendentry{S-$\llbracket \bu \rrbracket_\tau$};

        \addplot+[dashed, no marks] table[x index=0, y index=1, col sep=space] {data/dphim_mesh3_S.txt};
        \addlegendentry{S-$\phi$};

        \end{axis}
        \end{tikzpicture}
        (c)        
    \end{minipage}
\caption{(a) Matrix and fracture mean pressures, (b) matrix and fracture mean temperatures, and (c) scaled mean $\jump{\bu}_\tau$, $d_f-d_0$ and  $\phi-\phi^0$ as functions of time for the solutions on the finest mesh $m=3$, for the discretisations of the enthalpy-based  (H) and entropy-based (S) models. }
        \label{fig:meanVARIATIONSLiquid}    
\end{figure}

\begin{figure}[H]
    \centering

    \begin{tikzpicture}[scale=0.8]
        \node at (0,0) { };
        
        \draw[blue, thick, mark=*, mark options={blue}] plot coordinates {(1,0.1) (1.5,0.1)}; \node[right] at (1.5,0) {$p$};
        \draw[red, thick, mark=square*, mark options={red}] plot coordinates {(3,0.1) (3.5,0.1)}; \node[right] at (3.5,0) {$T$};
        \draw[green, thick, mark=triangle*, mark options={green}] plot coordinates {(5,0.1) (5.5,0.1)}; \node[right] at (5.5,0) {$\phi$};
        
        \draw[orange, thick, mark=*, mark options={orange}] plot coordinates {(7,0.1) (7.5,0.1)}; \node[right] at (7.5,0) {$d_f$};
        \draw[purple, thick, mark=square*, mark options={purple}] plot coordinates {(9,0.1) (9.5,0.1)}; \node[right] at (9.5,0) {$\llbracket \bu \rrbracket_\tau$};
        \draw[brown, thick, mark=triangle*, mark options={brown}] plot coordinates {(11,0.1) (11.5,0.1)}; \node[right] at (11.5,0) {$p_f$};
        \draw[violet, thick, mark=diamond*, mark options={violet}] plot coordinates {(13,0.1) (13.5,0.1)}; \node[right] at (13.5,0) {$T_f$};
    \end{tikzpicture}
    
    \begin{minipage}{0.48\textwidth}
        \begin{tikzpicture}[scale=0.8]
            \begin{loglogaxis}[
                title={Convergence Rates},
                xlabel={$\mathrm{h} \text { (1/total number of fracture faces) }$},
                ylabel={$L^2$ Error},
                grid=major
            ]
            \addplot[blue, mark=*] table[x index=0,y index=1,col sep=space] {data/convh_m.txt};
            \addplot[red, mark=square*] table[x index=0,y index=2,col sep=space] {data/convh_m.txt};
            \addplot[green, mark=triangle*] table[x index=0,y index=3,col sep=space] {data/convh_m.txt};
            \addplot[dashed, blue, mark=*] table[x index=0,y index=1,col sep=space] {data/convh_m_S.txt};
            \addplot[dashed, red, mark=square*] table[x index=0,y index=2,col sep=space] {data/convh_m_S.txt};
            \addplot[dashed, green, mark=triangle*] table[x index=0,y index=3,col sep=space] {data/convh_m_S.txt};
            
            \logLogSlopeTriangle{0.75}{0.3}{0.07}{1.5}{black};
            \end{loglogaxis}

        \end{tikzpicture}
    \end{minipage}
    \hfill
    \begin{minipage}{0.48\textwidth}
        \begin{tikzpicture}[scale=0.8]
            \begin{loglogaxis}[
                title={Convergence Rates},
                xlabel={$\mathrm{h} \text { (1/total number of fracture faces) }$},
                ylabel={},
                grid=major
            ]
            \addplot[orange, mark=*] table[x index=0,y index=1,col sep=space] {data/convh_f.txt};
            \addplot[purple, mark=square*] table[x index=0,y index=2,col sep=space] {data/convh_f.txt};
            \addplot[brown, mark=triangle*] table[x index=0,y index=3,col sep=space] {data/convh_f.txt};
            \addplot[violet, mark=diamond*] table[x index=0,y index=4,col sep=space] {data/convh_f.txt};

            \addplot[dashed, orange, mark=*] table[x index=0,y index=1,col sep=space] {data/convh_f_S.txt};
            \addplot[dashed, purple, mark=square*] table[x index=0,y index=2,col sep=space] {data/convh_f_S.txt};
            \addplot[dashed, brown, mark=triangle*] table[x index=0,y index=3,col sep=space] {data/convh_f_S.txt};
            \addplot[dashed, violet, mark=diamond*] table[x index=0,y index=4,col sep=space] {data/convh_f_S.txt};
            \logLogSlopeTriangle{0.75}{0.3}{0.07}{1.5}{black};
            \end{loglogaxis}
        \end{tikzpicture}
    \end{minipage}
    \caption{Relative $L^2$ in time errors vs.~the mesh step $h$ (1/total number of fracture faces) of mean in space $p$, $p_f$, $T$, $T_f$ $d_f$, $\jump{\bu}_\tau$, $\phi$  discrete solutions for both enthalpy-based (full lines) and entropy-based models (dash lines). The errors are computed using the fine mesh $m=3$ reference solutions with the same time stepping and model. }

    \label{fig:convergenceratesLiquid}
\end{figure}
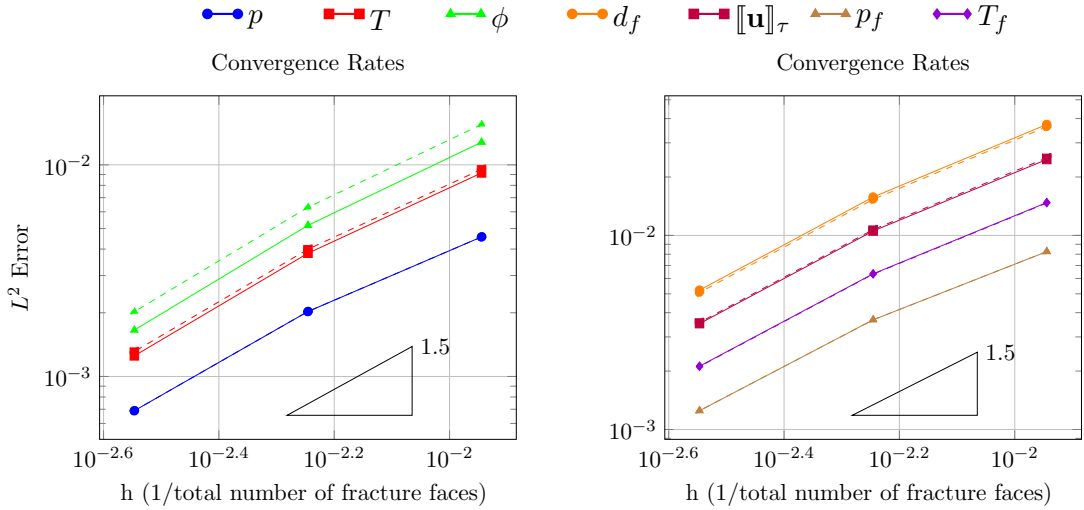

\subsubsection{Perfect gas case} \label{sec:testgas} 
We now consider the case of a perfect gas characterised by the following specific enthalpy and density 
\begin{equation*}
h(p,T) = C_f T, \quad \varrho(p,T) = M_f {p\over R T},
\end{equation*}
with gas specific heat capacity $C_f = 1000$ J.Kg$^{-1}$.K$^{-1}$, molar mass $M_f = 28.9645 ~10^{-3}$ Kg.mol$^{-1}$ and the perfect gas constant $R= 8.3149$ J.mol$^{-1}$.K$^{-1}$. 
The gas viscosity is fixed to $\eta = 1.72~10^{-5} \mathrm{~Pa} \mathrm{~s}$.

The time intervals for each of the three stages are given by $t^{(1)} = 100$ s, $t^{(2)} = 500$ s and $t_F = 3.005~10^5$ s. 
The time stepping is defined with a single time step during stage 1 since the coupling of the displacement with the pressure is very small during this stage due to the high gas compressibility (see the pressure solution during stage 1 in Figure \ref{fig:simuGaz} (a)). Stage 2 is initialised with a time step of $10$ s and the maximum time steps are set to $100$ s for stage 2 and $8000$ s for stage 3.

In Figure \ref{fig:convergence-rates-GAZ} we present the $L^2$-errors in time, vs.~the mesh size, for the spacial means of relevant variables. 
Both schemes provide roughly a convergence order of $1.5$. It can also be noticed that the enthalpy-based discrete solution is not very well captured by the coarser mesh, which explains the strong decrease of the errors between the first two meshes.
The significant differences between the convergence plots of the two models result from the high discrepancy between the enthalpy-based and entropy-based discrete solutions.

Figure \ref{fig:Mean-S-gaz} better illustrates the difference between the solutions of both models; we see in particular that the entropy-based model exhibits a temperature variation far below $-15$ K during stage 3 while it should physically be close to $15$ K as it is the case for the  enthalpy-based model. We demonstrate in Figure \ref{fig:Mean-S-Vnablap-gaz} that this discrepancy is due to the neglected terms $\bV_m\cdot \nabla p$ in the approximate entropy equations of the \eqref{model:entropy} model. Once these terms are added, the discrete solutions of both models fully match. 
This can be easily explained by comparing during stage 3 the order of magnitude of the neglected term $\mathbf{V}_m \cdot \nabla p$ with that of $\varrho_m \mathbf{V}_m \cdot \nabla h_m$. The ratio between both terms is of the order of roughly $10$ for the gas test case while it is roughly $0.1$ for the liquid test case. This is explained by the rather high pressure gradient combined with the low density and heat capacity in the gas case compared with the liquid case. We can conclude that the terms $\bV_m\cdot \nabla p$ cannot be neglected in the entropy-based model in the gas case with low specific density.

\begin{figure}[H]
    \centering
    \begin{tabular}{cc}
    \includegraphics[width=0.45\linewidth]{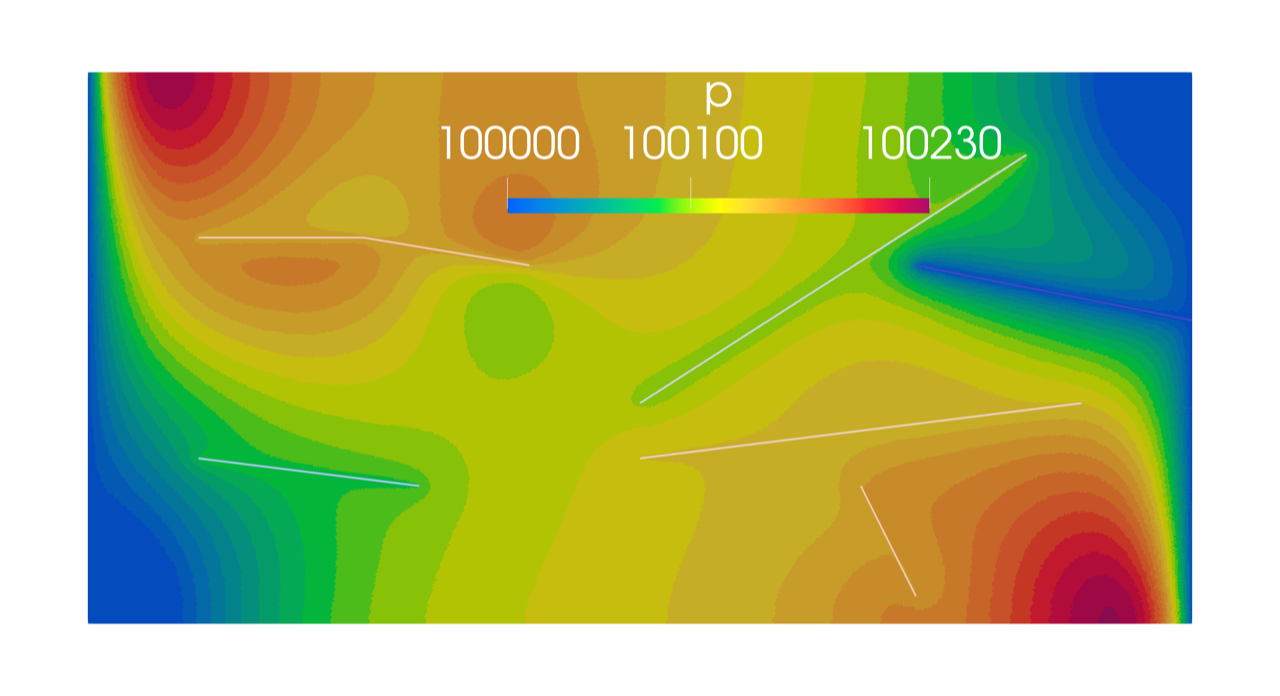}&
    \includegraphics[width=0.45\linewidth]{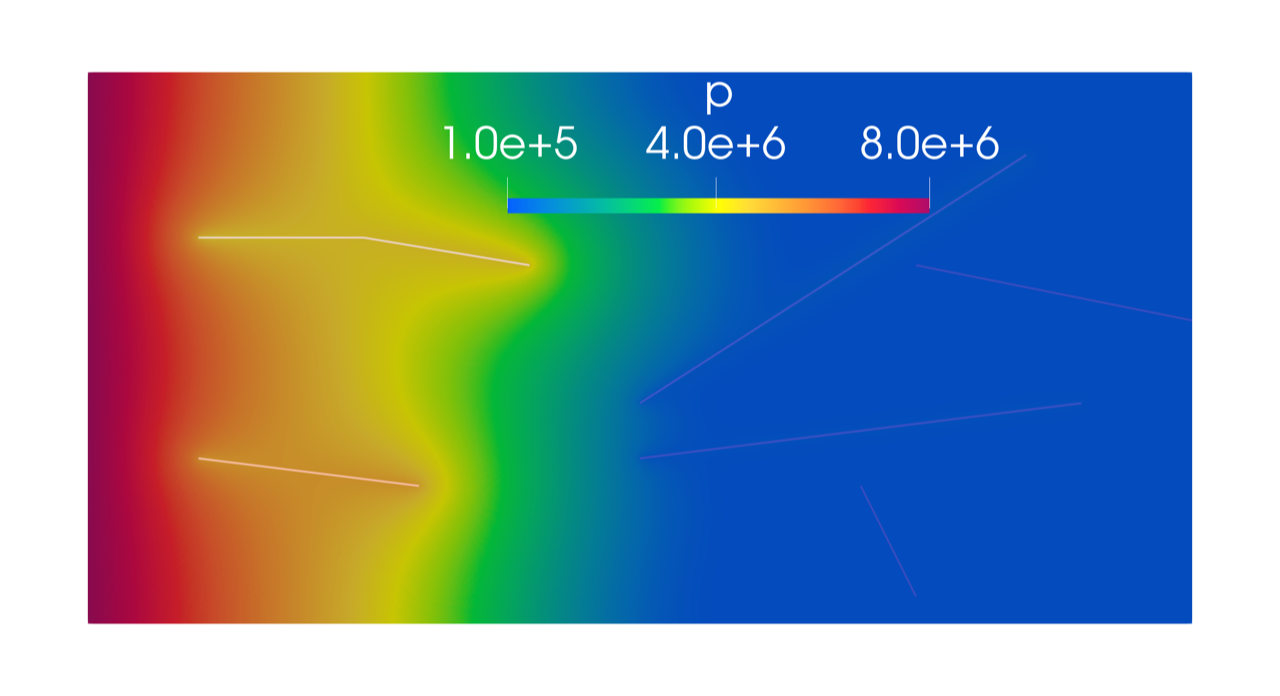}\\
      (a) Stage 1 & (b) Stage 2\\
    \multicolumn{2}{c}{
        \includegraphics[width=0.45\linewidth]{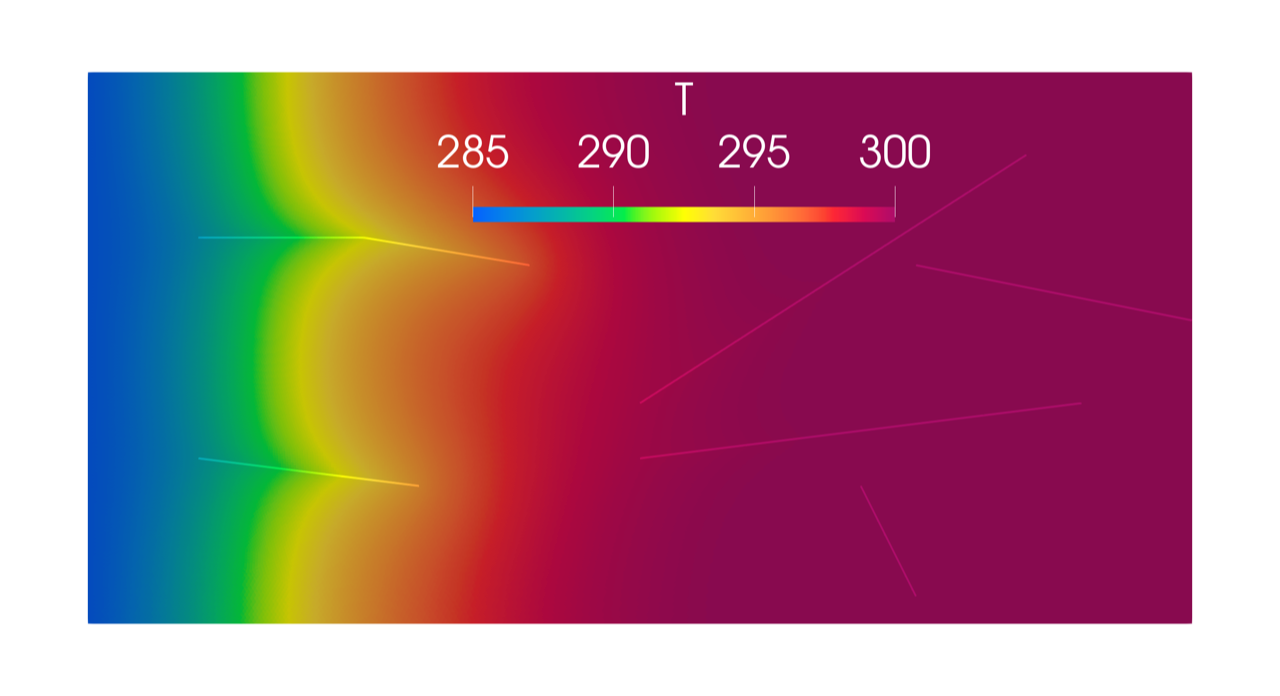}
        }\\
    \multicolumn{2}{c}{
        (c) Stage 3
      }
    \end{tabular}
    \caption{Discrete solution obtained with the enthalpy-based model on the finest mesh $m=3$: (a) pressure $p$ at the end of stage 1, (b) pressure $p$ at $t=170$ s during stage 2, (c) temperature $T$ at $t=42730$ s during stage 3.}
    \label{fig:simuGaz}
\end{figure}

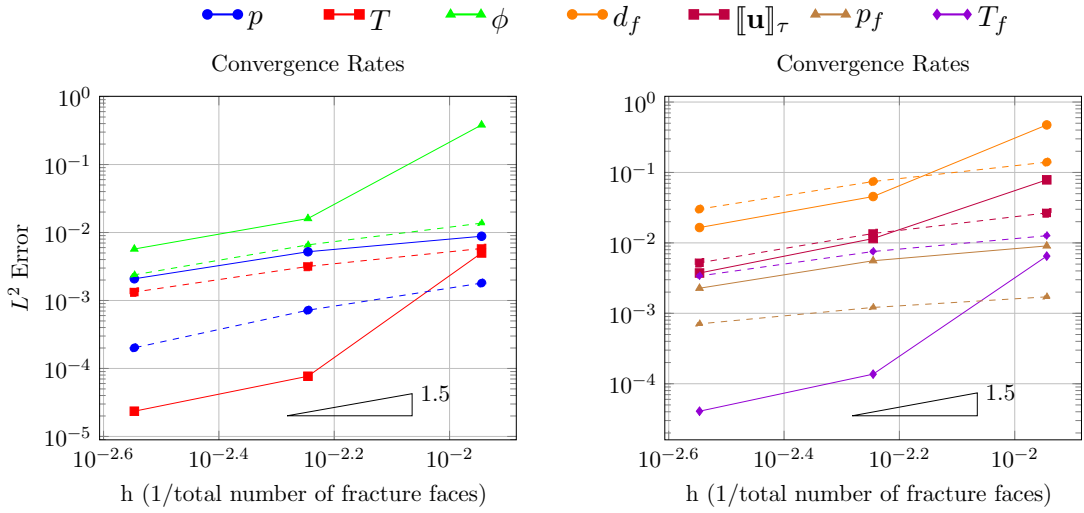
\begin{figure}[H]
    \centering
    
    \begin{tikzpicture}[scale=0.8]
        \node at (0,0) { };
        
        \draw[blue, thick, mark=*, mark options={blue}] plot coordinates {(1,0.1) (1.5,0.1)}; \node[right] at (1.5,0) {$p$};
        \draw[red, thick, mark=square*, mark options={red}] plot coordinates {(3,0.1) (3.5,0.1)}; \node[right] at (3.5,0) {$T$};
        \draw[green, thick, mark=triangle*, mark options={green}] plot coordinates {(5,0.1) (5.5,0.1)}; \node[right] at (5.5,0) {$\phi$};
        
        \draw[orange, thick, mark=*, mark options={orange}] plot coordinates {(7,0.1) (7.5,0.1)}; \node[right] at (7.5,0) {$d_f$};
        \draw[purple, thick, mark=square*, mark options={purple}] plot coordinates {(9,0.1) (9.5,0.1)}; \node[right] at (9.5,0) {$\llbracket \bu \rrbracket_\tau$};
        \draw[brown, thick, mark=triangle*, mark options={brown}] plot coordinates {(11,0.1) (11.5,0.1)}; \node[right] at (11.5,0) {$p_f$};
        \draw[violet, thick, mark=diamond*, mark options={violet}] plot coordinates {(13,0.1) (13.5,0.1)}; \node[right] at (13.5,0) {$T_f$};
    \end{tikzpicture}
    
    \begin{minipage}{0.48\textwidth}
        \begin{tikzpicture}[scale=0.8]
            \begin{loglogaxis}[
                title={Convergence Rates},
                xlabel={$\mathrm{h} \text { (1/total number of fracture faces) }$},
                ylabel={$L^2$ Error},
                grid=major
            ]
            \addplot[blue, mark=*] table[x index=0,y index=1,col sep=space] {data/convh_m_H-g.txt};
            \addplot[red, mark=square*] table[x index=0,y index=2,col sep=space] {data/convh_m_H-g.txt};
            \addplot[green, mark=triangle*] table[x index=0,y index=3,col sep=space] {data/convh_m_H-g.txt};

             \addplot[dashed, blue, mark=*] table[x index=0,y index=1,col sep=space] {data/convh_m_S-g.txt};
            \addplot[dashed, red, mark=square*] table[x index=0,y index=2,col sep=space] {data/convh_m_S-g.txt};
            \addplot[dashed, green, mark=triangle*] table[x index=0,y index=3,col sep=space] {data/convh_m_S-g.txt};
            \logLogSlopeTriangle{0.75}{0.3}{0.07}{1.5}{black};
            \end{loglogaxis}
        \end{tikzpicture}
    \end{minipage}
    \hfill
    \begin{minipage}{0.48\textwidth}
        \begin{tikzpicture}[scale=0.8]
            \begin{loglogaxis}[
                title={Convergence Rates},
                xlabel={$\mathrm{h} \text { (1/total number of fracture faces) }$},
                ylabel={},
                grid=major
            ]
            \addplot[orange, mark=*] table[x index=0,y index=1,col sep=space] {data/convh_f_H-g.txt};
            \addplot[purple, mark=square*] table[x index=0,y index=2,col sep=space] {data/convh_f_H-g.txt};
            \addplot[brown, mark=triangle*] table[x index=0,y index=3,col sep=space] {data/convh_f_H-g.txt};
            \addplot[violet, mark=diamond*] table[x index=0,y index=4,col sep=space] {data/convh_f_H-g.txt};

            \addplot[dashed, orange, mark=*] table[x index=0,y index=1,col sep=space] {data/convh_f_S-g.txt};
            \addplot[dashed, purple, mark=square*] table[x index=0,y index=2,col sep=space] {data/convh_f_S-g.txt};
            \addplot[dashed, brown, mark=triangle*] table[x index=0,y index=3,col sep=space] {data/convh_f_S-g.txt};
            \addplot[dashed, violet, mark=diamond*] table[x index=0,y index=4,col sep=space] {data/convh_f_S-g.txt};
            \logLogSlopeTriangle{0.75}{0.3}{0.07}{1.5}{black};
            \end{loglogaxis}
        \end{tikzpicture}
    \end{minipage}
    \caption{Relative $L^2$ in time errors vs.~the mesh step $h$ (1/total number of fracture faces) of mean in space $p$, $p_f$, $T$, $T_f$ $d_f$, $\jump{\bu}_\tau$, $\phi$  discrete solutions for both enthalpy-based (full lines) and entropy-based models (dash lines). The errors are computed using the fine mesh $m=3$ reference solutions with the same time stepping and model.}
    \label{fig:convergence-rates-GAZ}
\end{figure}


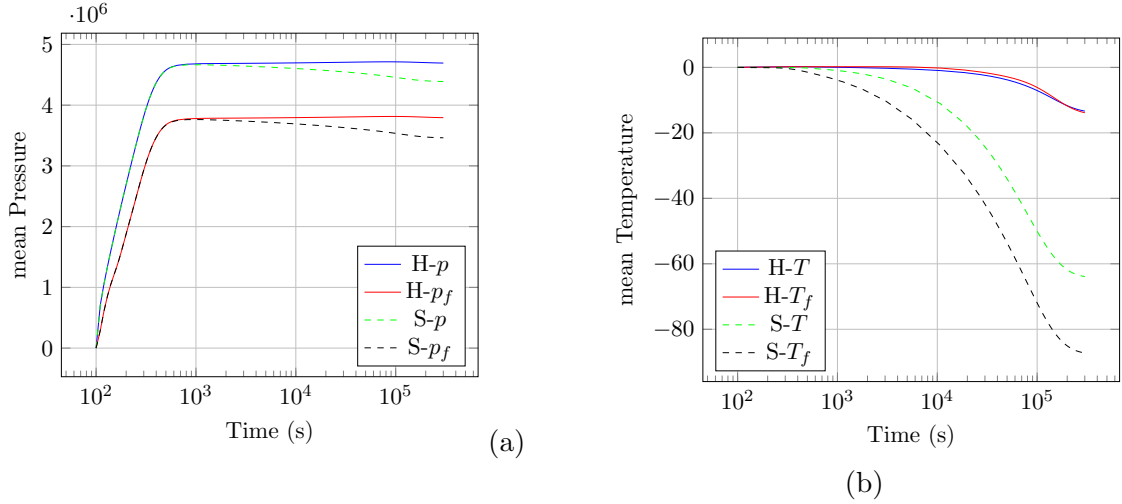
\begin{figure}[H]
    \centering
    
    \begin{minipage}{0.48\textwidth}
        \centering
        \begin{tikzpicture}[scale=0.8]
                   \begin{axis}[
            title={},
            xlabel={Time (s)},
            ylabel={mean Pressure},
            grid=major,
            xmode=log, 
           legend pos=south east
        ]
        
        \addplot+[no marks] table[x index=0, y index=2, col sep=space] {data/Tm_Pm_Phim-g.txt};
        \addlegendentry{H-$p$};
        
        \addplot+[no marks] table[x index=0, y index=2, col sep=space] {data/df_pf_sf_slip-g.txt};
        \addlegendentry{H-$p_f$};
        
        \addplot+[dashed, no marks, color=green] table[x index=0, y index=2, col sep=space] {data/Tm_Pm_Phim_S-g.txt};
        \addlegendentry{S-$p$};
        
        \addplot+[dashed, no marks] table[x index=0, y index=2, col sep=space] {data/df_pf_sf_slip_S-g.txt};
        \addlegendentry{S-$p_f$};
        
        \end{axis}
        \end{tikzpicture}
        (a) 
    \end{minipage}
   \hfill 
    \begin{minipage}{0.48\textwidth}
      \centering
      \vskip 1cm 
        \begin{tikzpicture}[scale=0.8]
               \begin{axis}[
            title={},
            xlabel={Time (s)},
            ylabel={mean Temperature},
            grid=major,
            xmode=log, 
            legend pos=south west 
        ]
        
        \addplot+[no marks] table[x index=0, y index=1, col sep=space] {data/Tm_Pm_Phim-g.txt};
        \addlegendentry{H-$T$};
        
        \addplot+[no marks] table[x index=0, y index=3, col sep=space] {data/df_pf_sf_slip-g.txt};
        \addlegendentry{H-$T_f$};
        
        \addplot+[dashed, no marks, color=green] table[x index=0, y index=1, col sep=space] {data/Tm_Pm_Phim_S-g.txt};
        \addlegendentry{S-$T$};
        
        \addplot+[dashed, no marks] table[x index=0, y index=3, col sep=space] {data/df_pf_sf_slip_S-g.txt};
        \addlegendentry{S-$T_f$};
        
        \end{axis}
        \end{tikzpicture}
        (b) 
    \end{minipage}

    \caption{
Matrix and fracture (a) mean over-pressures ($p-p^0$) and (b) mean over-temperatures ($T-T^0$) as functions of time for the solutions on the finest mesh $m=3$, for the discretisations of the enthalpy-based  (H) and entropy-based (S) models.}      
    \label{fig:Mean-S-gaz}
\end{figure}

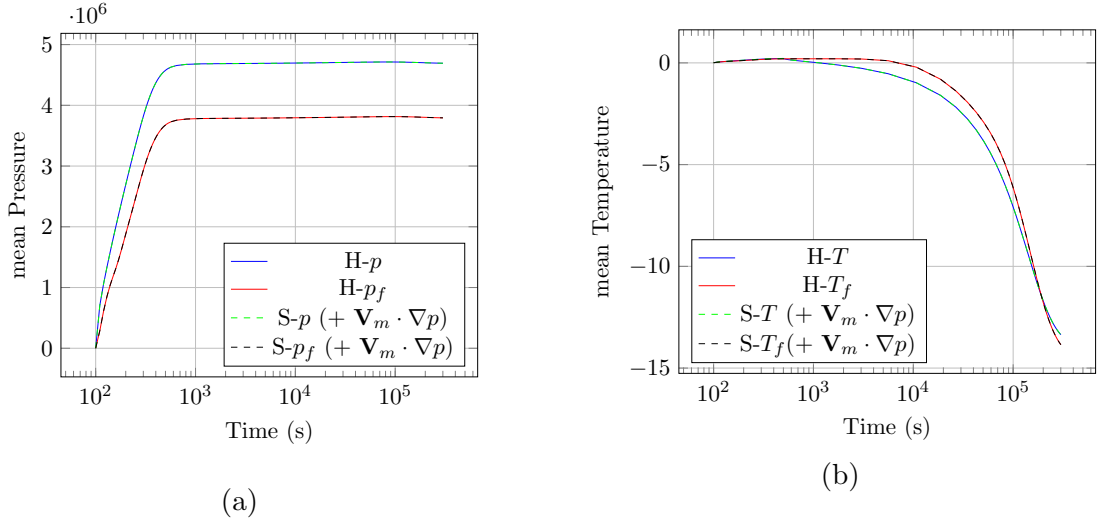
\begin{figure}[H]
    \centering
    
    \begin{minipage}{0.48\textwidth}
        \centering
        \begin{tikzpicture}[scale=0.8]
                   \begin{axis}[
            title={},
            xlabel={Time (s)},
            ylabel={mean Pressure},
            grid=major,
            xmode=log, 
           legend pos=south east
        ]
        
        \addplot+[no marks] table[x index=0, y index=2, col sep=space] {data/Tm_Pm_Phim-g.txt};
        \addlegendentry{H-$p$};
        
        \addplot+[no marks] table[x index=0, y index=2, col sep=space] {data/df_pf_sf_slip-g.txt};
        \addlegendentry{H-$p_f$};
        
        \addplot+[dashed, no marks, color=green] table[x index=0, y index=2, col sep=space] {data/Tm_Pm_Phim-g.txt};
        \addlegendentry{S-$p$ (+ $\bV_m\cdot \nabla p$)};
        
        \addplot+[dashed, no marks] table[x index=0, y index=2, col sep=space] {data/df_pf_sf_slip-g.txt};
        \addlegendentry{S-$p_f$ (+ $\bV_m\cdot \nabla p$)};
        
        \end{axis}
        \end{tikzpicture}
        \vskip 0.3cm      (a) 
    \end{minipage}
    \hfill
    \begin{minipage}{0.48\textwidth}
      \centering
        \begin{tikzpicture}[scale=0.8]
               \begin{axis}[
            title={},
            xlabel={Time (s)},
            ylabel={mean Temperature},
            grid=major,
            xmode=log, 
            legend pos=south west 
        ]
        
        \addplot+[no marks] table[x index=0, y index=1, col sep=space] {data/Tm_Pm_Phim-g.txt};
        \addlegendentry{H-$T$};
        
        \addplot+[no marks] table[x index=0, y index=3, col sep=space] {data/df_pf_sf_slip-g.txt};
        \addlegendentry{H-$T_f$};
        
        \addplot+[dashed, no marks, color=green] table[x index=0, y index=1, col sep=space] {data/Tm_Pm_Phim-g.txt};
        \addlegendentry{S-$T$ (+ $\bV_m\cdot \nabla p$)};
        
        \addplot+[dashed, no marks] table[x index=0, y index=3, col sep=space] {data/df_pf_sf_slip-g.txt};
        \addlegendentry{S-$T_f$(+ $\bV_m\cdot \nabla p$)};
        
        \end{axis}
        \end{tikzpicture}
        (b) 
    \end{minipage}

    \caption{
Matrix and fracture (a) mean over-pressures ($p-p^0$) and (b) mean over-temperatures ($T-T^0$) as functions of time for the solutions on the finest mesh $m=3$, for the discretisations of the enthalpy-based (H) model, and the entropy-based model with $\bV_m\cdot \nabla p$ correction (S ($+\bV_m\cdot \nabla p$)) discrete models.} 
    \label{fig:Mean-S-Vnablap-gaz}
\end{figure}

\subsubsection{Performances of the nonlinear solver}

Figures \ref{fig: Nk-liquid} and \ref{fig: Nk-Gaz} exhibit the total numbers of Newton iterations against time for the Thermo-Hydro and Mechanical models. The total number of time steps is $110$ for the liquid case in Figure 
\ref{fig: Nk-liquid} and  $87$ for the gas case in Figure \ref{fig: Nk-Gaz}. 
We can notice the robustness in both cases of the nonlinear solvers w.r.t.~the mesh size. Remarkably, in the liquid case, the entropy and enthalpy-based discrete models provide similar numbers of iterations through time.
In the gas case, the number of Thermo-Hydro Newton iterations gets moderately larger during stage 3 for the entropy-based than for the enthalpy-based discrete models  as a result of a much larger temperature variation for the entropy-based simulation. It has been checked that adding the $\bV_m\cdot \nabla p$ terms to the entropy-based discrete model gives back essentially the same Newton behavior as for the enthalpy-based model. 

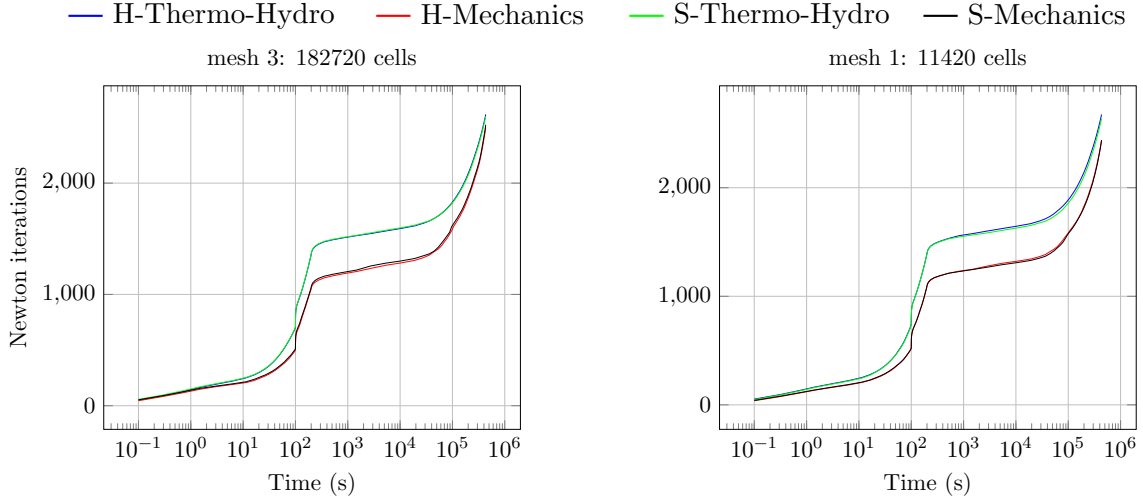
\begin{figure}[H]
    \centering

    \begin{tikzpicture}[scale=0.9]
        \node at (0,0) { };
        
        \draw[blue, thick] plot coordinates {(.5,0.1) (1,0.1)}; \node[right] at (1,0.1) {H-Thermo-Hydro};
      
        \draw[red, thick] plot coordinates {(5,0.1) (5.5,0.1)}; \node[right] at (5.5,0.13) {H-Mechanics };

        \draw[green, thick] plot coordinates {(8.7,0.1) (9.2,0.1)}; \node[right] at (9.2,0.1) {S-Thermo-Hydro};
        
        \draw[black, thick] plot coordinates {(13,0.1) (13.5,0.1)}; \node[right] at (13.5,0.13) {S-Mechanics};
    \end{tikzpicture}
    
    \begin{minipage}{0.45\textwidth}
        \centering
        \begin{tikzpicture}[scale=0.8]
                   \begin{axis}[
            title={mesh 3: 182720 cells},
            xlabel={Time (s)},
            ylabel={Newton iterations},
            grid=major,
            xmode=log, 
           legend pos=south east
        ]
        
        \addplot+[blue, no marks] table[x index=0, y index=1, col sep=space] {data/TotalNewtonFlowofT.txt};

        \addplot+[red, no marks] table[x index=0, y index=1, col sep=space] {data/TotalNewtonMecaofT.txt};

         \addplot+[green, no marks] table[x index=0, y index=1, col sep=space] {data/TotalNewtonFlowofT-S-3.txt};

        \addplot+[black, no marks] table[x index=0, y index=1, col sep=space] {data/TotalNewtonMecaofT-S-3.txt};

        \end{axis}
        \end{tikzpicture}
    \end{minipage}
    \hfill
    \begin{minipage}{0.45\textwidth}
        \centering
        \begin{tikzpicture}[scale=0.8]
               \begin{axis}[
            title={mesh 1: 11420 cells},
            xlabel={Time (s)},
            ylabel={},
            grid=major,
            xmode=log, 
            legend pos=south east 
        ]
        
        \addplot+[blue, no marks] table[x index=0, y index=1, col sep=space] {data/TotalNewtonFlowofT_mesh1.txt};

        \addplot+[red, no marks] table[x index=0, y index=1, col sep=space] {data/TotalNewtonMecaofT_mesh1.txt};

        \addplot+[green, no marks] table[x index=0, y index=1, col sep=space] {data/TotalNewtonFlowofT-S-1.txt};

        \addplot+[black, no marks] table[x index=0, y index=1, col sep=space] {data/TotalNewtonMecaofT-S-1.txt};

        \end{axis}
        \end{tikzpicture}
    \end{minipage}

    \caption{Total numbers of Newton iterations for the Thermo-Hydro and Mechanical models as a function of time, for the liquid case and both enthalpy-based (H) and entropy-based (S) schemes. (Left) mesh $m=3$, (right) mesh $m=1$, with a total number of $110$ time steps in all cases.}
    \label{fig: Nk-liquid}

\end{figure}

\begin{figure}[H]
    \centering

    \begin{tikzpicture}[scale=0.9]
        \node at (0,0) { };
        \draw[blue, thick] plot coordinates {(.5,0.1) (1,0.1)}; \node[right] at (1,0.1) {H-Thermo-Hydro};
      
        \draw[red, thick] plot coordinates {(5,0.1) (5.5,0.1)}; \node[right] at (5.5,0.13) {H-Mechanics };

        \draw[green, thick] plot coordinates {(8.7,0.1) (9.2,0.1)}; \node[right] at (9.2,0.1) {S-Thermo-Hydro};
        
        \draw[black, thick] plot coordinates {(13,0.1) (13.5,0.1)}; \node[right] at (13.5,0.13) {S-Mechanics};
    \end{tikzpicture}
    
    \begin{minipage}{0.45\textwidth}
        \centering
        \begin{tikzpicture}[scale=0.8]
                   \begin{axis}[
            title={mesh 3: 182720 cells},
            xlabel={Time (s)},
            ylabel={Newton iterations},
            grid=major,
            xmode=log, 
           legend pos=south east
        ]
        
        \addplot+[blue, no marks] table[x index=0, y index=1, col sep=space] {data/TotalNewtonFlowofT-H-3-g.txt};

        \addplot+[red, no marks] table[x index=0, y index=1, col sep=space] {data/TotalNewtonMecaofT-H-3-g.val.txt};

        \addplot+[green, no marks] table[x index=0, y index=1, col sep=space] {data/TotalNewtonFlowofT-S-3-g.txt};

        \addplot+[black, no marks] table[x index=0, y index=1, col sep=space] {data/TotalNewtonMecaofT-S-3-g.val.txt};

        \end{axis}
        \end{tikzpicture}
    \end{minipage}
    \hfill
    \begin{minipage}{0.45\textwidth}
        \centering
        \begin{tikzpicture}[scale=0.8]
               \begin{axis}[
            title={mesh 1: 11420 cells},
            xlabel={Time (s)},
            ylabel={},
            grid=major,
            xmode=log, 
            legend pos=south east 
        ]
        
        \addplot+[blue, no marks] table[x index=0, y index=1, col sep=space] {data/TotalNewtonFlowofT-H-1-g.txt};

        \addplot+[red, no marks] table[x index=0, y index=1, col sep=space] {data/TotalNewtonMecaofT-H-1-g.val.txt};

        \addplot+[green, no marks] table[x index=0, y index=1, col sep=space] {data/TotalNewtonFlowofT-S-1-g.txt};

        \addplot+[black, no marks] table[x index=0, y index=1, col sep=space] {data/TotalNewtonMecaofT-S-1-g.val.txt};

        \end{axis}
        \end{tikzpicture}
    \end{minipage}

    \caption{Total number of Newton iterations for the Thermo-Hydro and Mechanical models as a function of time, for the gas case and both enthalpy-based (H) and entropy-based (S) schemes. (Left) mesh $m=3$, (right) mesh $m=1$, with a total number of $87$ time steps in all cases.}
   \label{fig: Nk-Gaz}

\end{figure}
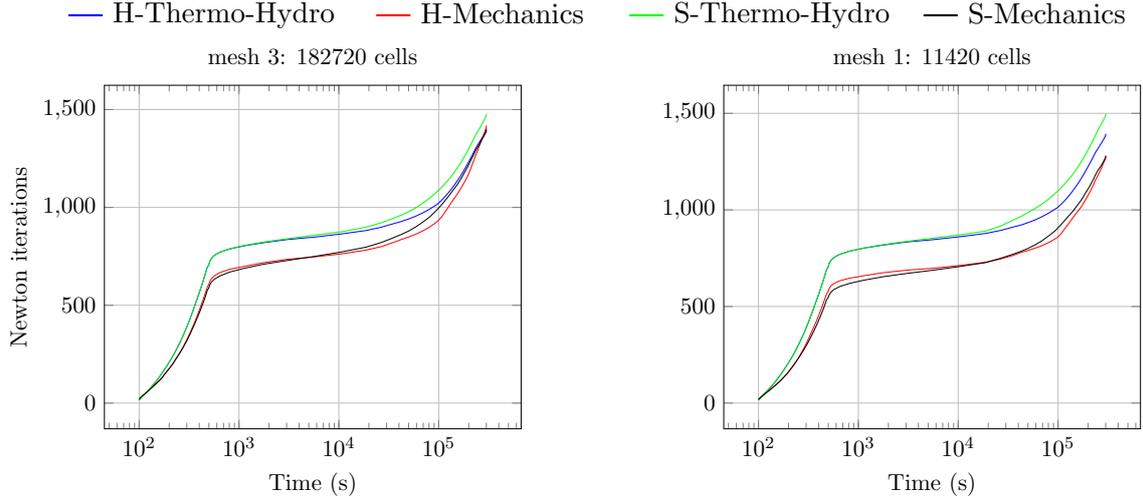

\section{Conclusion}

This work focuses on discretisations of mixed-dimensional THM models preserving energy estimates for a general single phase fluid thermodynamical model. Our approach uses a Finite Volume scheme  for the mass and energy equations with a possible upwinding of the convection terms in order to account for convection dominated regimes. It is combined with a mixed Finite Element discretisation of the contact-mechanics with face-wise constant Lagrange multipliers, keeping  the dissipative property of the contact terms at matrix fracture interfaces. 
Two formulations of the energy equation are considered and compared numerically. It is built either directly from the energy conservation or obtained from an approximate entropy balance equation based on a small Darcy velocity and small temperature variation assumptions. The Finite Volume discretisation of the entropy-based model, and in particular of the non-conservative convection terms, is carefully designed in order to preserve the link between both formulations, which in turns guarantees that it satisfies an energy estimate.
Both discrete models are assessed and compared in terms of convergence, accuracy and robustness on 2D test cases including a convective dominated regime, and either a weakly compressible liquid or highly compressible gas. It is shown that both approaches provides similar results in terms of spatial convergence and robustness of the nonlinear solver. On the other hand, in the gas case, for low specific density and high pressure gradient, the terms $\bV\cdot \nabla p$ that are typically neglected in the entropy balance approach must be accounted for in order to provide the physical solution.  

\vskip 0.5cm 
\noindent{\bf Acknowledgements}: the authors are  grateful to Andra and BRGM for partially funding this work and to Laurence Beaude, Marc Leconte, Simon Lopez, Antoine Pasteau and Farid Smai for fruitful discussions during the elaboration of this work. 

J. Droniou would like to acknowledge a partial funding by the European Union (ERC Synergy, NEMESIS, project number 101115663).
Views and opinions expressed are however those of the authors only and do not necessarily reflect those of the European Union or the European Research Council Executive Agency. Neither the European Union nor the granting authority can be held responsible for them.

%
%

\bibliographystyle{plain}
\bibliography{ThermoPoroMecanique.bib}
\end{document}